\documentclass[11pt]{amsart}
\usepackage{amsmath,mathtools}
\usepackage{graphicx, color, enumerate}
\usepackage{amscd}
\usepackage{amsmath}
\usepackage{amsfonts}
\usepackage{amssymb}
\usepackage{mathrsfs}
\usepackage{latexsym}
\usepackage[colorlinks=true]{hyperref}
\usepackage[foot]{amsaddr}

\newcommand\e\varepsilon

\newcommand\R{\mathbb R}

\newcommand\de\partial
\newcommand\weakto\rightharpoonup

\renewcommand\le\leqslant
\renewcommand\ge\geqslant
\renewcommand\a\alpha

\renewcommand\b\beta

\renewcommand\d\delta
\newcommand\vfi\varphi
\newcommand\g\gamma
\newcommand\gb\gamma
\renewcommand\l\lambda
\newcommand\n\nabla
\newcommand\s\sigma
\renewcommand\t\theta
\renewcommand\O\S

\newcommand\G\Gamma

\renewcommand\S\Sigma
\renewcommand\L\Lambda

\renewcommand\o\S

\def\bbm[#1]{\text{\boldmath $#1$}}

\renewcommand\leq{\leqslant}

\newtheorem{theorem}{Theorem}[section]
\newtheorem{lemma}[theorem]{Lemma}

\newtheorem{definition}[theorem]{Definition}

\newtheorem{remark}[theorem]{Remark}

\def\sideremark#1{\ifvmode\leavevmode\fi\vadjust{\vbox to0pt{\vss
\hbox to0pt{\hskip\hsize\hskip1em
\vbox{\hsize3cm\tiny\raggedright\pretolerance10000
\noindent #1\hfill}\hss}\vbox to8pt{\vfil}\vss}}}

\definecolor{edu}{rgb}{0,1,0.2}

\numberwithin{equation}{section}

\title[Nonlinear elliptic systems involving Hardy-Sobolev Criticalities]
{Nonlinear elliptic systems involving Hardy-Sobolev Criticalities}


\author{Rafael L\'opez-Soriano, Alejandro Ortega}

\email[Rafael López-Soriano]{ralopezs@ugr.es}%
\email[Alejandro Ortega ]{alortega@math.uc3m.es}
\address[R. López-Soriano,]{Departamento de Análisis Matem\'atico,
Universidad de Granada, Campus Fuentenueva, 18071 Granada, Spain}
\address[A. Ortega]{Departamento de Matem\'aticas,
Universidad Carlos III de Madrid, Av. Universidad 30, 28911 Legan\'es (Madrid), Spain}

\begin{document}
\maketitle

\begin{abstract}
This paper is focused on the solvability of a family of nonlinear elliptic systems defined in $\mathbb{R}^N$. Such equations contain Hardy potentials and Hardy--Sobolev criticalities coupled by a possible critical Hardy--Sobolev term. That problem arises as a generalization of Gross--Pitaevskii and Bose--Einstein type systems. By means of variational techniques, we shall find ground and bound states in terms of the coupling parameter $\nu$ and the order of the different parameters and exponents. In particular, for a wide range of parameters we find solutions as minimizers or Mountain--Pass critical points of the energy functional on the underlying Nehari manifold. 
\end{abstract}

\section{Introduction}\setcounter{equation}0
In this work, we study the existence of positive bound and ground states for an elliptic system involving critical Hardy--Sobolev terms, namely
\begin{equation}\label{system:alphabeta}
\left\{\begin{array}{ll}
\displaystyle -\Delta u - \lambda_1 \frac{u}{|x|^2}-\frac{u^{2^*_{s}-1}}{|x|^{s}}= \nu \alpha h(x) \frac{u^{\alpha-1}v^\beta}{|x|^{s}}   &\text{in }\mathbb{R}^N,\vspace{.3cm}\\
\displaystyle -\Delta v - \lambda_2 \frac{v}{|x|^2}-\frac{v^{2^*_{s}-1}}{|x|^{s}}= \nu \beta h(x) \frac{u^\alpha v^{\beta-1}}{|x|^{s}} &\text{in }\mathbb{R}^N,\vspace{.3cm}\\
u,v> 0 & \text{in }\mathbb{R}^N\setminus\{0\},
\end{array}\right.
\end{equation}
where $\lambda_1,\lambda_2\in(0,\Lambda_N)$ being $\Lambda_N=\frac{(N-2)^2}{4}$ the best constant in the Hardy's inequality, $0< s<2$, $\nu>0$ and  $\alpha,\beta$ are real parameters such that 
\begin{equation}\label{alphabeta}\tag{$\alpha\beta$}
\alpha, \beta> 1 \qquad \mbox{ and } \qquad  \alpha+\beta\le2^*_{s}.
\end{equation}
The value $2^*_{s}=\frac{2(N-s)}{N-2}$ denotes the critical Hardy--Sobolev exponent, whereas $h\gneq 0$ is a function such that
\begin{equation}\label{H1}\tag{H}
h(x)\in L^{\frac{2^*_s}{2^*_s-\alpha-\beta},s}(\mathbb{R}^N),
\end{equation}
where $L^{p,s}(\mathbb{R}^N)$ denotes the weighted $L^p$-space of measurable functions $u$ such that
\[\| u\|_{L^{p,s}(\mathbb{R}^N)}^p:= \int_{\mathbb{R}^N}\frac{|u|^p}{|x|^s}dx<\infty.\]
In the case of having $\alpha+\beta=2^*_s$, then \eqref{H1} simply establishes $h\in L^{\infty}(\mathbb{R}^N)$.

The aim of this work is to establish some results concerning the existence of bound and ground states of \eqref{system:alphabeta}, extending the results given in \cite{CoLSOr} for the underlying system with critical Sobolev terms instead of Hardy--Sobolev ones. In particular, taking $s=0$, system \eqref{system:alphabeta} reads as the nonlinear elliptic system 
\begin{equation}\label{system:alphabeta_0}
\left\{\begin{array}{ll}
\displaystyle -\Delta u - \lambda_1 \frac{u}{|x|^2}-u^{2^*-1}= \nu \alpha h(x) u^{\alpha-1}v^\beta   &\text{in }\mathbb{R}^N,\vspace{.3cm}\\
\displaystyle -\Delta v - \lambda_2 \frac{v}{|x|^2}-v^{2^*-1}= \nu \beta h(x) u^\alpha v^{\beta-1} &\text{in }\mathbb{R}^N,\vspace{.3cm}\\
u,v> 0 & \text{in }\mathbb{R}^N\setminus\{0\},
\end{array}\right.
\end{equation}
where $2^*=2_0^*=\frac{2N}{N-2}$ is the classical critical Sobolev exponent. In \cite{CoLSOr} we proved the existence of positive bound and ground states in terms of the size of the coupling parameter $\nu$ and the values of $\alpha,\beta,\lambda_1,\lambda_2$. 

The system \eqref{system:alphabeta} is closely related to a system of coupled nonlinear Schr\"odinger equations, typically known as Gross--Pitaevskii like system. Such problem  arises, for instance, in the modelling in several physical phenomena, such as the Hartree--Fock theory for a double condensate, binary mixture Bose--Einstein condensates or nonlinear optics,  (cf. \cite{Esry, Frantz,Akhmediev,Kivshar} and references therein). In particular, if one looks for solitary-wave solutions to the above-mentioned Gross--Pitaevski system, one gets 
\begin{equation}\label{BSsystem}
\left\{\begin{array}{ll}
\displaystyle{-\Delta u + V_1(x) u= \mu_1 u^{3} + \nu u v^{2}}  &\text{in }\mathbb{R}^N,\vspace{.3cm}\\
-\Delta v + V_2(x) v = \mu_2 v^{3} + \nu u^{2}v &\text{in }\mathbb{R}^N,
\end{array}\right.
\end{equation}
for some potentials $V_1(x)$  and $V_2(x)$. We refer to \cite{CoLSOr} for specific details. Moreover, \eqref{BSsystem} can be seen as a particular case of the Schrödinger system
\begin{equation}\label{pBSsystem}
\left\{\begin{array}{ll}
\displaystyle{-\Delta u + V_1(x) u= \mu_1 u^{2p-1} + \nu u^{p-1} v^{p}}  &\text{in }\mathbb{R}^N,\vspace{.3cm}\\
-\Delta v + V_2(x) v = \mu_2 v^{2p-1} + \nu u^{p}v^{p-1} &\text{in }\mathbb{R}^N,
\end{array}\right.
\end{equation}
where $1<p\le\frac{N}{N-2}$ with $N\ge 3$. In the subcritical regime, namely $p<\frac{N}{N-2}$, the question of the existence and multiplicity of solutions for \eqref{pBSsystem} has been extensively analyzed under some assumptions on $V_j$ and $\nu$, see \cite{AC2, BW, LinWei, MMP, POMP}, among others. Concerning the critical regime with constant potentials, system \eqref{pBSsystem} admits only the trivial solution $(u,v)=(0,0)$ due to a Pohozaev--type identity. Introducing Hardy type potentials this situation changes dramatically, giving rise to the existence of nontrivial bound and ground states. \newline 
If one considers \eqref{pBSsystem} with critical Hardy--Sobolev terms and general couplings then system \eqref{system:alphabeta} arises. Finally, let us mention that taking $\alpha=2$ and $\beta=1$, the problem \eqref{system:alphabeta_0} exhibits some particularities with respect to the case \eqref{alphabeta}. That problem is usually known as a Schr\"odinger--Korteweg--De Vries type system. More details about solvability and compactness can be found in \cite{CoLSOr2}.

\

Systems with Hardy-Sobolev critical terms have not been much studied in the literature.  For that reason, as a first step, along this paper we shall study the influence of such criticalities supposing that they have the same order. That assumption allows us to control naturally coupled terms by using the decoupled ones, see the proof of Theorem~\ref{thmsemitrivialalphabeta}, for instance. Some results regarding the existence of solutions are available for bounded domains, see \cite{NyJa,ZhaZho}.  Actually, the order of the singularities in the Hardy-Sobolev terms may be different. For that case one can consider some embedding between the underlying functional spaces, which is not possible in entire space. Concerning the existence of solutions, we expect to give some results generalizing the hypothesis on the Hardy-Sobolev orders. As far as we know, this is the first paper where such problem is considered in $\mathbb{R}^N$. 

\
Let us point out that another novelty in this work is to highlight the role of the function $h$. Assumption \eqref{H1} guarantees certain homogeneity between the coupling and the decoupled terms of system \eqref{system:alphabeta}. In a certain sense, this hypothesis serves to compensate a possible lack of criticality of the coupling term. For the case $s=0$, studied in \cite{AbFePe,CoLSOr,CoLSOr2}, it is assumed that $h\in L^1(\mathbb{R}^N)\cap L^\infty(\mathbb{R}^N)$ if $\alpha+\beta<2^*$, which  in turn also implies that $h\in L^q(\mathbb{R}^N)$ for every $q>1$. Our approach relaxes such condition by choosing the appropriate intermediate functional space that still allows $h$ to control de concentration phenomena at $0$ and $\infty$, namely, $h\in L^{\frac{2^*}{2^*-\alpha-\beta}}(\mathbb{R}^N)$ if $s=0$ or $h\in L^{\frac{2_s^*}{2_s^*-\alpha-\beta},s}(\mathbb{R}^N)$ if $s>0$.
 
\

Our main goal is then to prove, by means of a variational approach, the existence of positive solutions to the system \eqref{system:alphabeta} in terms of the different parameters. In particular, we will look for solutions as critical points of the associated energy functional
\begin{equation}\label{functalphabeta}
\begin{split}
\mathcal{J}_\nu (u,v)=&\frac{1}{2} \int_{\mathbb{R}^N} \left( |\nabla u|^2 + |\nabla v|^2  \right) \, dx -\frac{\lambda_1}{2} \int_{\mathbb{R}^N} \dfrac{u^2}{|x|^2} \, dx   -\frac{\lambda_2}{2} \int_{\mathbb{R}^N} \dfrac{v^2}{|x|^2} \, dx  \vspace{0,7cm} \\
&- \frac{1}{2^*_{s}} \int_{\mathbb{R}^N} \frac{|u|^{2^*_{s}}}{|x|^{s}} \, dx - \frac{1}{2^*_{s}} \int_{\mathbb{R}^N} \frac{|v|^{2^*_{s}}}{|x|^{s}} \, dx -\nu \int_{\mathbb{R}^N} h(x) \frac{|u|^\alpha |v|^{\beta}}{|x|^{s}} \, dx  ,
\end{split}
\end{equation}
defined in the the Sobolev space $\mathbb{D}=\mathcal{D}^{1,2} (\mathbb{R}^N)\times \mathcal{D}^{1,2}(\mathbb{R}^N)$ with $\mathcal{D}^{1,2}(\mathbb{R}^N)$ defined as the completion of $C^{\infty}_0(\mathbb{R}^N)$ under the norm
\begin{equation*}
\|u\|_{\mathcal{D}^{1,2}(\mathbb{R}^N)}^2=\int_{\mathbb{R}^N} \, |\nabla u|^2  \, dx  .
\end{equation*}
A standard strategy is then to localize critical points of $\mathcal{J}_\nu$ on the underlying Nehari manifold where $\mathcal{J}_\nu$ is bounded. On the other hand, this variational approach also requires some compactness properties. This feature is given through a Palais--Smale condition, relying on the well known \textit{concentration-compactness principle}, cf. \cite{Lions1,Lions2}. Essentially, at this step the related difficulties rely on the lack of compactness of the map
\begin{equation*} 
u\mapsto \dfrac{u}{|x|^{\frac{s}{2^*_s}}},
\end{equation*} 
from $\mathcal{D}^{1,2}(\mathbb{R}^N)$ into $L^{q}(\mathbb{R}^N)$ when $q=2^*_s$, while this embedding is compact if $q<2^*_s$, cf. \cite[Lemma~3.2]{GhYu}. Then, the nonlinear coupling term, $\dfrac{u^\alpha v^\beta}{|x|^s}$, may be critical depending on the value of $\alpha+\beta$. We distinguish thus between the subcritical regime $\alpha+\beta< 2^*_{s}$, where the compactness follows by the above-mentioned compact embedding, and the critical regime $\alpha+\beta= 2^*_{s}$, that requires a more careful analysis. In this critical setting, the following hypotheses will be assumed
\begin{equation}\label{H}\tag{H0}
h \mbox{ is continuous around $0$ and $\infty$ and }h(0)=\lim_{x\to+\infty} h(x)=0.
\end{equation}
Section~\ref{section:PS} is then devoted to the analysis of the Palais--Smale condition under a {\it quantization} of the energy levels of $\mathcal{J}_{\nu}$.

Next, observe that, for any $\nu \in \mathbb{R}$, the system \eqref{system:alphabeta} admits two \textit{semi-trivial} positive solutions $(z_1,0)$ and $(0,z_2)$, where $z_j$ satisfies the entire problem
\begin{equation*}
-\Delta z_j - \lambda_j \frac{z_j}{|x|^2}=\frac{z_j^{2^*_{s}-1}}{|x|^{s}} \qquad \mbox{ and } \qquad z_j>0 \qquad \mbox{ in } \mathbb{R}^N\setminus \{ 0\}.
\end{equation*}
Notice that this problem is invariant under the scaling $z^{(j)}_\mu (x) =\mu^{-\frac{N-2}{2}} z_j\left(\frac{x}{\mu} \right)$. The explicit expression of $z_j$ was found by Kang and Peng in \cite{KangPeng}, which is recalled in Section~\ref{section2} jointly with several qualitative properties. From now on, we denote the \textit{semi-trivial} energy levels as
\begin{equation*}
\mathcal{J}_\nu(z^{(1)}_\mu,0)=\mathfrak{C}(\lambda_1,s) \qquad \mathcal{J}_\nu(0,z^{(2)}_\mu)=\mathfrak{C}(\lambda_2,s),
\end{equation*}
where $\mathfrak{C}(\lambda,s)$ is defined in subsection~\ref{subsection:semitrivials} (see \eqref{critical_levels} below).
\

These \textit{semi-trivial} solutions will play an important role in our analysis. As a first step we provide a characterization of the \textit{semi-trivial} solutions as critical points of $\mathcal{J}_\nu$ in Theorem~\ref{thmsemitrivialalphabeta}. More precisely, we will prove that the couples $(z^{(1)}_\mu,0)$ and $(0,z^{(2)}_\mu)$ become either a local minimum or a saddle point of $\mathcal{J}_\nu$ on the corresponding Nehari manifold under some hypotheses on the parameters $\nu,\alpha,\beta$. Such classification allow us to study the geometry of the functional $\mathcal{J}_\nu$ and to obtain some energy estimates, crucial to deduce existence of solutions.\newline 
The coupling parameter $\nu$ and the exponents $\alpha,\beta$  have a subtle effect in the geometry of the functional $\mathcal{J}_\nu$. If $\nu$ is large enough, the minimum energy level is strictly lower than that of the \textit{semi-trivial} solutions. Then, one can find a positive ground state by minimizing.
\begin{theorem}\label{thm:nugrande}
Assume \eqref{H1} and either $\alpha+\beta<2^*_s$ or $\alpha+\beta=2^*_s$ satisfying \eqref{H}. Then there exists $\overline{\nu}>0$ such that the system \eqref{system:alphabeta} has a positive ground state $(\tilde{u},\tilde{v}) \in \mathbb{D}$ for every $\nu>\overline{\nu}$.
\end{theorem}
Since $\mathfrak{C}(\lambda,s)$ is decreasing in $\lambda$, the order between the energy of the \textit{semi-trivial} solutions depends on that of $\lambda_1, \lambda_2$. Indeed, if $\lambda_1\ge \lambda_2$, then $\mathfrak{C}(\lambda_1,s)\le \mathfrak{C}(\lambda_2,s)$ and the minimum energy corresponds to $(z^{(1)}_\mu,0)$, which is a saddle point under some assumptions on $\beta$ and $\nu$. Alternatively, if $\lambda_1<\lambda_2$, the minimum energy corresponds to $(0,z^{(2)}_\mu)$ which may be a saddle point depending on $\alpha$ and $\nu$. Both situations provide the existence of a positive ground state.
\begin{theorem}\label{thm:lambdaground}
Assume \eqref{H1} and either $\alpha+\beta<2^*_s$ or $\alpha+\beta=2^*_s$ satisfying \eqref{H}. If one of the following statements is satisfied
\begin{itemize}
\item[$i)$] $\lambda_1\ge \lambda_2$ and either $\beta=2$ and $\nu$ large enough or $\beta<2$,
\item[$ii)$] $\lambda_1\le \lambda_2$ and either $\alpha=2$ and $\nu$ large enough or $\alpha<2$,
\end{itemize}
then system \eqref{system:alphabeta} admits a positive ground state $(\tilde{u},\tilde{v})\in\mathbb{D}$.

In particular, if $\max\{\alpha,\beta\}<2$ or $\max\{\alpha,\beta\}\le2$ with $\nu$ sufficiently large, then system \eqref{system:alphabeta} admits a positive ground state $(\tilde{u},\tilde{v})\in\mathbb{D}$.
\end{theorem}
Next, we analyze the reverse situation. If $\lambda_1>  \lambda_2$, the minimum energy among the {\it semi-trivial} solutions is that of $(z^{(1)}_\mu,0)$. Indeed, if either $\beta>2$ or $\beta=2$ with $\nu$ small enough, $(z^{(1)}_\mu,0)$ is a local minimum. An analogous situation holds for ($0,z^{(2)}_\mu)$. Under that hypotheses, there exists $\nu>0$ small such a {\it semi-trivial} solution becomes a ground state of \eqref{system:alphabeta}.
\begin{theorem}\label{thm:groundstatesalphabeta}
Assume \eqref{H1} and either $\alpha+\beta<2^*_s$ or $\alpha+\beta=2^*_s$ satisfying \eqref{H}. Then, 
\begin{itemize}
\item[$i)$] If $\alpha\ge 2$ and {$\lambda_1<\lambda_2$}, then there exists $\tilde{\nu}>0$ such that for any $0<\nu<\tilde{\nu}$  the couple $(0,z_\mu^{(2)})$ is the ground state of \eqref{system:alphabeta}.
\item[$ii)$]  If $\beta\ge 2$ and {$\lambda_1>\lambda_2$}, then there exists $\tilde{\nu}>0$ such that for any $0<\nu<\tilde{\nu}$  the couple $(z_\mu^{(1)},0)$ is the ground state of \eqref{system:alphabeta}.
\item[$iii)$] In particular, if $\alpha,\beta \ge 2$, then there exists $\tilde{\nu}>0$ such that for any $0<\nu<\tilde{\nu}$, the couple $(0, z_\mu^{(2)})$ is a ground state of \eqref{system:alphabeta} if ${\lambda_1<\lambda_2}$, whereas $(z_\mu^{(1)},0)$ is a ground state otherwise.
\end{itemize}
\end{theorem}
Finally, we found bound states by using a min-max procedure. More precisely, we show that the energy functional admits a Mountain--Pass geometry for certain $\lambda_1,\lambda_2$ verifying a separability condition, that allows us to separate the \textit{semi-trivial} energy levels in a suitable way. The geometry of $\mathcal{J}_\nu$ jointly with the PS condition give us the existence of solution.
\begin{theorem}\label{MPgeom}
Assume \eqref{H1} and either $\alpha+\beta<2^*_s$ or $\alpha+\beta=2^*_s$ satisfying \eqref{H}. If
\begin{itemize}
\item[$i)$] Either $\alpha \ge 2$ and
\begin{equation}\label{lamdasalphabeta}
2\mathfrak{C}(\lambda_2,s)> \mathfrak{C}(\lambda_1,s)>  \mathfrak{C}(\lambda_2,s),
\end{equation}
\item[$ii)$] or $\beta \ge 2 $ and
\begin{equation}\label{lamdasalphabeta2}
2\mathfrak{C}(\lambda_1,s)> \mathfrak{C}(\lambda_2,s)>  \mathfrak{C}(\lambda_1,s),
\end{equation}
\end{itemize}
then there exists $\tilde{\nu}>0$ such that for $0<\nu\le \tilde{\nu}$, the system \eqref{system:alphabeta} admits a bound state given as a Mountain--Pass-type critical point.
\end{theorem}

\begin{remark}
It is immediate using the definition of $\mathfrak{C}(\lambda,s)$, see \eqref{critical_levels}, that \eqref{lamdasalphabeta} is equivalent to the condition
$$
\lambda_2>\lambda_1 \qquad \mbox{ and } \qquad  \dfrac{\Lambda_N-\lambda_2}{\Lambda_N-\lambda_1}>2^{-\frac{2(2-s)}{2(N-1)-s}},
$$
whereas \eqref{lamdasalphabeta2} can be rewritten in an analogous way.
\end{remark}

A key step in our approach is the careful application of Lemma~\ref{algelemma} below to the mass of the component whose exponent is greater than $2$. Precisely, this result allows to derive lower bounds on integral terms and, consequently, to prove that the critical mass of the underlying component must not vanish. For that reason some of our existence results deal with $\max \{\alpha,\beta\}\ge 2$. The proofs of the above theorems are deferred to Section \ref{section:main}.
\section{Preliminaries and Functional setting}\label{section2}
In this section, we introduce the appropriate variational setting for the system \eqref{system:alphabeta}. The problem \eqref{system:alphabeta} is the Euler--Lagrange system for the energy functional $\mathcal{J}_\nu$ (see \eqref{functalphabeta}) which is correctly defined in the product space $\mathbb{D}=\mathcal{D}^{1,2} (\mathbb{R}^N)\times \mathcal{D}^{1,2} (\mathbb{R}^N)$. The energy space $\mathbb{D}$ is equipped with the norm
\begin{equation*}
\|(u,v)\|^2_{\mathbb{D}}=\|u\|^2_{\lambda_1}+\|v\|^2_{\lambda_2},
\end{equation*}
where
\begin{equation*}
\|u\|^2_{\lambda}=\int_{\mathbb{R}^N} |\nabla u|^2 \, dx - \lambda \int_{\mathbb{R}^N} \frac{u^2}{|x|^2} \, dx.
\end{equation*}
Because of the Hardy's inequality,
\begin{equation}\label{hardy_inequality}
\Lambda_N \int_{\R^N} \frac{u^2}{|x|^2} \, dx \leq \int_{\R^N} |\nabla u|^2 \, dx,
\end{equation}
the norm $\|\cdot\|_{\lambda}$ is equivalent to $\|\cdot\|_{\mathcal{D}^{1,2} (\mathbb{R}^N)}$ for any $\lambda\in (0,\Lambda_N)$, where $\Lambda_N=\frac{(N-2)^2}{4}$ is the best constant in the Hardy inequality.

\

As commented before, the equations arising from the decoupled system, namely $\nu=0$, are well known nowadays. In particular, if either the system is decoupled, i.e. $\nu=0$, or some component vanishes, $u$ or $v$ satisfies the entire equation
\begin{equation}\label{entire}
-\Delta z - \lambda_j \frac{z}{|x|^2}=\frac{z^{2^*_s-1}}{|x|^{s}} \qquad\mbox{ and }   \qquad z>0\ \mbox{ in } \mathbb{R}^N \setminus \{0\}.
\end{equation}
A complete classification of \eqref{entire} is given in \cite{KangPeng} where it is proved that, if $\lambda\in\left(0,\Lambda_N\right)$, the solutions of \eqref{entire} are given by
\begin{equation}\label{zeta}
z_{\mu}^{(j)}(x)= \mu^{-\frac{N-2}{2}}z_1^{\lambda_j,s}\left(\frac{x}{\mu}\right) \qquad \mbox{ with } \qquad z_1^{\lambda,s}(x)=\dfrac{A(N,\lambda)^{\frac{N-2}{2(2-s)}}}{|x|^{a_{\lambda}}\left(1+|x|^{(2-s)(1-\frac{2}{N-2}a_{\lambda})}\right)^{\frac{N-2}{2-s}}},
\end{equation}
where $\displaystyle A(N,\lambda)=2(\Lambda_N-\lambda)\frac{N-s}{\sqrt{\Lambda_N}}$,  $a_\lambda=\sqrt{\Lambda_N}-\sqrt{\Lambda_N-\lambda}$ and $\mu>0$ is a scaling factor. Solutions of \eqref{entire} arise as minimizers of the underlying Rayleigh quotient
\begin{equation*}
\mathcal{S}(\lambda,s)= \inf_{\substack{u\in \mathcal{D}^{1,2}(\mathbb{R}^N)\\
u\not\equiv0}}\frac{\|u\|^2_{\lambda}}{\left(\displaystyle\int_{\mathbb{R}^N}\frac{u^{2^*_s}}{|x|^s}dx\right)^{\frac{2}{2_s^*}}}= \frac{\|z_\mu^{\lambda,s}\|^2_{\lambda}}{\left(\displaystyle\int_{\mathbb{R}^N}\frac{(z_{\mu}^{\lambda,s})^{2^*_s}}{|x|^s}dx\right)^{\frac{2}{2_s^*}}}.
\end{equation*}
Actually, it holds that
\begin{equation}\label{normcrit}
\displaystyle  \|z_\mu^\lambda\|_{\lambda}^{2} = \displaystyle\int_{\mathbb{R}^N}\frac{(z_{\mu}^{\lambda,s})^{2^*_s}}{|x|^s}dx=[ \mathcal{S}(\lambda,s)]^{\frac{N-s}{2-s}}.
\end{equation}
Taking $u(x)=|x|^{a_\lambda}z_1^{\lambda,s}$, the equation \eqref{entire} becomes (in a weak sense) 
\begin{equation}\label{transformada}
-\text{div}(|x|^{-2a_\lambda}\nabla u)=\frac{u^{2_s^*-1}}{|x|^{ 2_s^*a_\lambda+s}}.
\end{equation}
Because of  \cite[Theorem A]{ChouChu} with $\beta_\lambda=-2 a_\lambda$ and $\alpha_{\lambda,s}=-(2_s^*a_\lambda+s)$, and noticing that $2_s^*=\dfrac{2(N+\alpha_{\lambda,s})}{N+\beta_\lambda-2}$, we get   
\begin{equation}\label{Slambda}
\mathcal{S}(\lambda,s)=4(\Lambda_N-\lambda)\frac{N-s}{N-2}\left(\frac{N-2}{2(2-s)\sqrt{\Lambda_N-\lambda}}\frac{2\pi^{\frac{N}{2}}}{\Gamma\left(\frac{N}{2}\right)}\frac{\Gamma^2\left(\frac{N-s}{2-s}\right)}{\Gamma\left(\frac{2(N-s)}{2-s}\right)}\right)^{\frac{2-s}{N-s}}.
\end{equation}
Observe that the constant $\mathcal{S}(\lambda,s)$ is decreasing in both $\lambda$ and $s$, so that $\mathcal{S}(0,0)\ge\mathcal{S}(\lambda,s)$. By definition, $\mathcal{S}(\lambda,s)$ is the best constant for the inequality 
\begin{equation}\label{H-S_lambda}
\mathcal{S}(\lambda,s)\left(\int_{\R^N} \frac{u^{2_s^*}}{|x|^s} \, dx \right)^{\frac{2}{2_s^*}} \leq\int_{\mathbb{R}^N} |\nabla u|^2 \, dx - \lambda \int_{\mathbb{R}^N} \frac{u^2}{|x|^2} \, dx.
\end{equation}
Moreover, by \eqref{transformada}, the constant $\mathcal{S}(\lambda,s)$ turns out to be also the best constant in the following Caffarelli--Kohn--Nirenberg-type inequality,
\begin{equation*}
\mathcal{S}(\lambda,s) \left(\int_{\R^N} |x|^{\alpha_{\lambda,s}}u^{2_s^*} \, dx \right)^{\frac{2}{2_s^*}}\leq \int_{\R^N} |x|^{\beta_{\lambda}}|\nabla u|^2 \, dx.
\end{equation*}
Since $\beta_0=\alpha_{0,0}=0$, we have $\mathcal{S}(0,0)=\mathcal{S}$, where $\mathcal{S}=\pi N(N-2)\left(\frac{\Gamma\left(\frac{N}{2}\right)}{\Gamma(N)}\right)^{\frac{2}{N}}$ is the best constant in Sobolev's inequality,
\begin{equation}\label{sobolev_inequality}
\mathcal{S}\left(\int_{\R^N} |u|^{2^*}dx\right)^{\frac{2}{2^*}} \leq \int_{\R^N} |\nabla u|^2dx.
\end{equation}
Since $\alpha_{\lambda,0}=-\frac{N}{\sqrt{\Lambda_N}}a_\lambda$, we have $\mathcal{S}(\lambda,0)=\left(1-\frac{\lambda}{\Lambda_N} \right)^{\frac{N-1}{N}}\mathcal{S}$. 
Finally, as $\beta_{0}=0$ and $\alpha_{0,s}=-s$,
\begin{equation*}
\mathcal{S}(0,s)=[(N-s)(N-2)]\left(\frac{1}{2-s}\frac{2\pi^{\frac{N}{2}}}{\Gamma\left(\frac{N}{2}\right)}\frac{\Gamma^2\left(\frac{N-s}{2-s}\right)}{\Gamma\left(\frac{2(N-s)}{2-s}\right)}\right)^{\frac{2-s}{N-s}}.
\end{equation*}
The constant $\mathcal{S}(0,s)$ is the best constant in the Hardy--Sobolev inequality,
\begin{equation}\label{hardy_sobolev_inequality}
\mathcal{S}(0,s) \left(\int_{\R^N} \frac{u^{2_s^*}}{|x|^s} \, dx \right)^{\frac{2}{2_s^*}}\leq \int_{\R^N} |\nabla u|^2 \, dx.
\end{equation}
In addition, one can see that (cf. \cite{ChouChu}), $\displaystyle\lim\limits_{s\to2^-}\mathcal{S}(0,s)=\Lambda_N$. 

For any $\mu>0$, the pairs $(z_{\mu}^{(1)},0)$ and $(0,z_{\mu}^{(2)})$ satisfying \eqref{system:alphabeta} will be referred to as \textit{semi-trivial} solutions. Our main aim is to look for solutions neither \textit{semi-trivial} nor trivial solutions, i.e., solutions $(u,v)$ such that $u\not\equiv 0$ and $v\not\equiv 0$ in $\mathbb{R}^N$.
\begin{definition}
We said that $(u,v)\in\mathbb{D}\setminus \{(0,0)\}$ is a non-trivial bound state for \eqref{system:alphabeta} if it is a non-trivial critical point of $\mathcal{J}_\nu$. 
A non-trivial and non-negative bound state $(\tilde{u},\tilde{v})$ is said to be a ground state if its energy is minimal, namely
\begin{equation}\label{ctilde}
\tilde{c}_\nu\vcentcolon=\mathcal{J}_\nu(\tilde{u},\tilde{v})=\min\left\{\mathcal{J}_\nu(u,v): (u,v)\in \mathbb{D}\setminus \{(0,0)\},\; u,v\ge 0 \mbox{ and } \mathcal{J}_\nu'(u,v)=0\right\}.
\end{equation}
\end{definition}
Note that the functional $\mathcal{J}_\nu$ is of class $C^1(\mathbb{D},\mathbb{R})$ and it is not bounded from below. Indeed,
\begin{equation*}
\mathcal{J}_\nu(t \tilde u,t \tilde v) \to -\infty \quad \mbox{ as } t\to\infty,
\end{equation*}
for $(\tilde u,\tilde v)\in \mathbb{D}\setminus\{(0,0)\}$. Next, we introduce a suitable constraint to minimize $\mathcal{J}_\nu$. Let us set 
\begin{equation*}
\begin{split}
\Psi(u,v)&=\left\langle \mathcal{J}'_\nu(u,v){\big|}(u,v)\right\rangle\\
&=\|(u,v)\|_\mathbb{D}^2-\int_{\mathbb{R}^N}\frac{|u|^{2^*_{s}}}{|x|^{s}}dx - \int_{\mathbb{R}^N}\frac{|v|^{2^*_{s}}}{|x|^{s}} dx -\nu (\alpha+\beta) \int_{\mathbb{R}^N} \frac{|u|^{\alpha} |v|^{\beta}}{|x|^s} \, dx,
\end{split}
\end{equation*}
and define the Nehari manifold associated to $\mathcal{J}_\nu$ as
\begin{equation*}
\mathcal{N}_\nu=\left\{ (u,v) \in \mathbb{D} \setminus \{(0,0)\} \, : \, \Psi(u,v) =0 \right\}.
\end{equation*}
Plainly, $\mathcal{N}_\nu$ contains all the non-trivial critical points of $\mathcal{J}_\nu$ in $\mathbb{D}$. Let us now recall some properties on $\mathcal{N}_\nu$ that will be of use throughout this work. Any $(u,v) \in \mathcal{N}_\nu$ satisfies
\begin{equation} \label{Nnueq1}
 \|(u,v)\|_\mathbb{D}^2=\int_{\mathbb{R}^N}\frac{|u|^{2^*_{s}}}{|x|^{s}}dx + \int_{\mathbb{R}^N}\frac{|v|^{2^*_{s}}}{|x|^{s}} dx +\nu (\alpha+\beta) \int_{\mathbb{R}^N} h \frac{|u|^{\alpha} |v|^{\beta}}{|x|^s} \, dx,
\end{equation}
so we can write the energy functional constrained to $\mathcal{N}_\nu$ as
\begin{equation}\label{Nnueq2}
\mathcal{J}_{\nu}{\big|}_{\mathcal{N}_\nu} (u,v) = \frac{2-s}{2(N-s)} \left(\int_{\mathbb{R}^N}  \frac{|u|^{2^*_{s}}}{|x|^{s}}+\frac{|v|^{2^*_{s}}}{|x|^{s}}dx\right) +\nu \left( \frac{\alpha+\beta-2}{2} \right)\int_{\mathbb{R}^N} h \frac{|u|^{\alpha} |v|^{\beta}}{|x|^s} dx.
\end{equation}
Given $(u,v)\in\mathbb{D}\setminus \{(0,0)\}$, there exists a unique value $t=t_{(u,v)}$ such that $(tu,tv) \in \mathcal{N}_\nu$. Indeed, $t$ is the unique solution to the algebraic equation
\begin{equation}\label{normH}
\|(u,v)\|_\mathbb{D}^2=\ t^{2_{s}^*-2}\left(\int_{\mathbb{R}^N} \frac{|u|^{2^*_{s}}}{|x|^{s}}dx + \int_{\mathbb{R}^N} \frac{|v|^{2^*_{s}}}{|x|^{s}} dx\right)+ \nu (\alpha+\beta) \, t^{\alpha+\beta-2} \int_{\mathbb{R}^N} h \frac{ |u|^{\alpha} |v|^{\beta}}{|x|^s}dx.
\end{equation}
Using \eqref{Nnueq1} together with \eqref{alphabeta} we find that, for any $(u,v) \in \mathcal{N}_\nu$,
\begin{equation}\label{criticalpoint1}
\begin{split}
\mathcal{J}_\nu''(u,v)[u,v]^2=&\ (2-\alpha-\beta)\|(u,v)\|_{\mathbb{D}}^2 \\
&+ (\alpha+\beta-2_{s}^*)\left(\int_{\mathbb{R}^N}  \frac{|u|^{2^*_{s}}}{|x|^{s}}dx + \int_{\mathbb{R}^N}  \frac{|v|^{2^*_{s}}}{|x|^{s}} dx\right)< 0.
\end{split}
\end{equation}
Moreover, $(0,0)$ is a strict minimum since, for the second variation of the energy functional, 
\begin{equation*}
 \mathcal{J}_\nu''(0,0)[\varphi_1,\varphi_2]^2=\|(\varphi_1,\varphi_2)\|^2_{\mathbb{D}}>0 \quad \text{ for any } (\varphi_1,\varphi_2)\in \mathcal{N}_\nu.
\end{equation*}
Hence, $(0,0)$ is an isolated point respect to $\displaystyle \mathcal{N}_\nu  \, \cup \, \{(0,0)\}$. As a consequence, $\mathcal{N}_\nu$ is a smooth complete manifold of codimension $1$. Also, there exists  $r_\nu>0$ such that
\begin{equation}\label{criticalpoint2}
\|(u,v)\|_{\mathbb{D}} > r_\nu\quad\text{for all } (u,v)\in \mathcal{N}_\nu.
\end{equation}

Given $(u,v) \in \mathbb{D}$ a critical point of $\mathcal{J}_{\nu}{\big|}_{\mathcal{N}_\nu}$, there exists a Lagrange multiplier $\omega$ such that
\begin{equation*}
(\mathcal{J}_{\nu}{\big|}_{\mathcal{N}_\nu})'(u,v)=\mathcal{J}'_\nu(u,v)-\omega \Psi'(u,v)=0.
\end{equation*}
Then, it follows that $\left\langle \mathcal{J}'_\nu(u,v){\big|}(u,v)\right\rangle = \omega \mathcal{J}_\nu''(u,v)[u,v]^2 $. By \eqref{criticalpoint1}, we deduce that $\omega=0$ and $\mathcal{J}'_\nu(u,v)=0$. Consequently, $\mathcal{N}_\nu$ is a called a natural constraint in the sense that
\vspace{0.15cm}
\begin{center}
$(u,v) \in \mathbb{D}$ is a critical point of $\mathcal{J}_\nu$ $\quad\Leftrightarrow\quad$ $(u,v) \in \mathbb{D}$ is a critical point of $\mathcal{J}_{\nu}{\big|}_{\mathcal{N}_\nu}$.
\end{center}
\vspace{0.15cm}
Let us also stress that, on the Nehari manifold $\mathcal{N}_\nu$,
\begin{equation}\label{Nnueq}
\begin{split}
\mathcal{J}_{\nu}{\big|}_{\mathcal{N}_\nu} (u,v) =& \left( \frac{1}{2}-\frac{1}{\alpha+\beta} \right) \|(u,v)\|^2_{\mathbb{D}} \\
&+ \left( \frac{1}{\alpha+\beta}-\frac{1}{2^*_{s}} \right)\left(\int_{\mathbb{R}^N} \frac{|u|^{2^*_{s}}}{|x|^{s}}dx+\int_{\mathbb{R}^N} \frac{|v|^{2^*_{s}}}{|x|^{s}} dx\right).
\end{split}
\end{equation}
From \eqref{criticalpoint2} and hypotheses \eqref{alphabeta} we have
\begin{equation*}
\mathcal{J}_{\nu} (u,v) > \left( \frac{1}{2}-\frac{1}{\alpha+\beta} \right) r^2_\nu\quad  \quad\text{for all } (u,v)\in \mathcal{N}_\nu.
\end{equation*}
Thus, $\mathcal{J}_{\nu}$ remains bounded from below on $\mathcal{N}_\nu$ and, hence, we can find solutions of \eqref{system:alphabeta} as minimizers of $\mathcal{J}_{\nu}{\big|}_{\mathcal{N}_\nu}$.
\subsection{Semi-trivial solutions}\label{subsection:semitrivials}\hfill\newline
Let us consider the decoupled energy functionals $\mathcal{J}_j:\mathcal{D}^{1,2} (\mathbb{R}^N)\mapsto\mathbb{R}$,
\begin{equation}\label{funct:Ji}
\mathcal{J}_j(u) =\frac{1}{2} \int_{\mathbb{R}^N}  |\nabla u|^2 \, dx -\frac{\lambda_j}{2} \int_{\mathbb{R}^N} \dfrac{u^2}{|x|^2}  \, dx - \frac{1}{2^*_{s}} \int_{\mathbb{R}^N} \frac{|u|^{2^*_{s}}}{|x|^{s}} \, dx.
\end{equation}
Note that
\begin{equation*}
\mathcal{J}_\nu(u,v)=\mathcal{J}_1(u)+\mathcal{J}_2(v)-\nu \int_{\mathbb{R}^N} h(x) \frac{|u|^{\alpha} |v|^{\beta}}{|x|^{s}} \, dx.
\end{equation*}
The function $z_{\mu}^{(j)}$, defined in \eqref{zeta}, is a global minimum of $\mathcal{J}_j$ on the Nehari manifold
\begin{equation}\label{Neharij}
\begin{split}
\mathcal{N}_j&= \left\{ u \in \mathcal{D}^{1,2} (\mathbb{R}^N) \setminus \{0\} \, : \,  \left\langle \mathcal{J}'_j(u){\big|} u\right\rangle=0 \right\}\\
&= \left\{ u \in\mathcal{D}^{1,2} (\mathbb{R}^N) \setminus \{0\} \, : \,  \|u\|_{\lambda_j}=\int_{\mathbb{R}^N} \frac{|u|^{2^*_{s}}}{|x|^{s}} \, dx \right\}.
\end{split}
\end{equation}
By using \eqref{normcrit}, one can compute the energy levels of $z_\mu^{(j)}$, namely,  for any $\mu>0$ we have
\begin{equation}\label{critical_levels}
\mathfrak{C}(\lambda_j,s)\vcentcolon=\mathcal{J}_j(z_\mu^{(j)})=\dfrac{2-s}{2(N-s)}\left[\mathcal{S}(\lambda_j,s)\right]^{\frac{N-s}{2-s}}.
\end{equation}
Then, the energy levels of the \textit{semi-trivial} solutions are given by
\begin{equation}\label{Jzeta}
\mathcal{J}_\nu(z_\mu^{(1)},0)=\mathfrak{C}(\lambda_1,s) \qquad\text{and}\qquad \mathcal{J}_\nu(0,z_\mu^{(2)})=\mathfrak{C}(\lambda_2,s).
\end{equation}
Let us remark that, since $\mathcal{S}(\lambda,s)$ is decreasing in both $\lambda$ and $s$, we have 
\begin{equation}\label{decreasing}
\mathfrak{C}(0,0)\ge\mathfrak{C}(\lambda,s)\qquad\text{for }\lambda\in(0,\Lambda_N)\ \text{and }s\in(0,2).
\end{equation}

Next, we characterize the variational nature of the \textit{semi-trivial} couples on $\mathcal{N}_\nu$. Although the proof follows as that of \cite[Theorem 2.2]{AbFePe} we include it to illustrate the effect of the parameter $s$ in the critical structure of system \eqref{system:alphabeta}.
\begin{theorem}\label{thmsemitrivialalphabeta} Under hypotheses \eqref{alphabeta} and \eqref{H1}, the following holds:
\begin{enumerate}
\item[i)] If $\alpha>2$ or $\alpha=2$ and $\nu$ small enough, then $(0,z_\mu^{(2)})$ is a local minimum of $\mathcal{J}_\nu$ on $\mathcal{N}_\nu$.
\item[ii)] If $\beta>2$ or $\beta=2$ and $\nu$ small enough, then $(z_\mu^{(1)},0)$ is a local minimum of $\mathcal{J}_\nu$ on $\mathcal{N}_\nu$.
\item[iii)] If $\alpha<2$ or $\alpha=2$ and $\nu$ large enough, then $(0,z_\mu^{(2)})$ is a saddle point for $\mathcal{J}_\nu$ on $\mathcal{N}_\nu$.
\item[iv)] If $\beta<2$ or $\beta=2$ and $\nu$ large enough, then $(z_\mu^{(1)},0)$ is a saddle point for $\mathcal{J}_\nu$ on $\mathcal{N}_\nu$.
\end{enumerate}
\end{theorem}
\begin{proof}
Note that $i)$ and $ii)$ can be proved in the same way, so let us focus on $i)$. In order to prove that $(0,z_\mu^{(2)})$ is a local minimum let us introduce a new couple $(\varphi,z_\mu^{2}+\psi)\in\mathcal{N}_\nu$ with $\mu>0$ and $\alpha>2$. We shall prove that the energy level of such perturbation is bigger than the semitrivial one. First, by \eqref{Nnueq1}, we have 
\begin{equation}\label{Nnueq1bis}
\begin{split}
\|(\varphi,z_\mu^{2}+\psi)\|^2_{\mathbb{D}}=&\int_{\mathbb{R}^N} \frac{|\varphi|^{2^*_{s}}}{|x|^{s}} + \int_{\mathbb{R}^N}\frac{|z_\mu^{(2)}+\psi|^{2^*_{s}}}{|x|^{s}} dx + (\alpha+\beta)\nu \int_{\mathbb{R}^N} h \frac{|\varphi|^{\alpha} |z_\mu^{(2)}+\psi|^{\beta}}{|x|^s} dx.
\end{split}
\end{equation}
Next, take $\overline{t}$ such that $\overline{t}(z_\mu^{(2)}+\psi)\in \mathcal{N}_{2}$, where $\mathcal{N}_{2}$ was defined in \eqref{Neharij}. Actually, by using the definition of $\mathcal{N}_2$ and \eqref{Nnueq1bis}, the value of $t$ is determined by the following expression
{\tiny
\begin{equation}\label{tminmax}
\begin{split}
 \overline{t}&=\left(\dfrac{\|z_\mu^{(2)}+\psi\|^2_{\lambda_2}}{\displaystyle{\int_{\mathbb{R}^N}\frac{|z_\mu^{(2)}+\psi|^{2^*_{s}}}{|x|^{s}}}dx} \right)^{\frac{1}{2^*_{s}-2}}= \left( 1-\dfrac{\|\varphi\|^2_{\lambda_1}-\displaystyle\int_{\mathbb{R}^N} \frac{|\varphi|^{2^*_{s}}}{|x|^{s}} \, dx - (\alpha+\beta)\nu \int_{\mathbb{R}^N} h \frac{|\varphi|^{\alpha} |z_\mu^{(2)}+\psi|^{\beta}}{|x|^s} }{\displaystyle\int_{\mathbb{R}^N}\dfrac{|z_\mu^{(2)}+\psi|^{2^*_{s}}}{|x|^{s}}} \right)^{\frac{1}{2^*_{s}-2}}.
\end{split}
\end{equation}
}
Then, using twice the Holder's inequality with exponents $p=\dfrac{2_{s}^*}{2_{s}^*-\alpha-\beta}$ and  $q=\dfrac{2_{s}^*}{\alpha+\beta}$ and $\tilde p = \dfrac{2_{s}^*}{\alpha}$ and $\tilde{q} = \dfrac{2_{s}^*}{\beta}$ respectively, we have
\begin{equation}\label{Hold1}
\begin{split}
 \int_{\mathbb{R}^N} h(x) \frac{|\varphi|^{\alpha} |z_\mu^{(2)}+\psi|^{\beta}}{|x|^s}dx & \leq \left(\int_{\mathbb{R}^N} \dfrac{h^{\frac{2_{s}^*}{2_{s}^*-\alpha-\beta}}}{|x|^{s}} \right)^{\frac{2_{s}^*-\alpha-\beta}{2_{s}^*}} \left(\int_{\mathbb{R}^N} \dfrac{(|\varphi|^\alpha |z_\mu^{(2)}+\psi|^\beta)^{\frac{2_{s}^*}{\alpha+\beta}}}{|x|^s} \right)^\frac{\alpha+\beta}{2_{s}^*} \\ 
&\leq C(h)  \left( \int_{\mathbb{R}^N} \frac{|\varphi|^{2^*_{s}}}{|x|^{s}} dx\right)^{\frac{\alpha}{2^*_{s}}}  \left( \int_{\mathbb{R}^N} \frac{|z_\mu^{(2)}+\psi|^{2^*_{s} }}{|x|^{s}} dx\right)^{\frac{\beta}{2^*_{s}}},
\end{split}
\end{equation}
where we have used \eqref{H1}. Notice that if $\alpha+\beta=2^*_s$, then $C(h)=||h||_{L^\infty(\mathbb{R}^N)}$.

From \eqref{Hold1} and \eqref{hardy_sobolev_inequality}, we find 
\begin{equation}\label{Hold2}
\int_{\mathbb{R}^N} h\frac{|\varphi|^{\alpha} |z_\mu^{(2)}+\psi|^{\beta}}{|x|^s}dx \leq  C \|\varphi \|^{\alpha}_{\lambda_1}.
\end{equation}
In particular, since $\alpha>2$, from \eqref{tminmax} and \eqref{Hold2}, we infer
\begin{equation}\label{texpan1}
\overline{t}^{2}= 1-\frac{2}{2^*_{s}-2}\frac{\|\varphi \|^{2}_{\lambda_1}(1+o(1))}{\displaystyle\int_{\mathbb{R}^N}\dfrac{|z_\mu^{(2)}+\psi|^{2^*_{s}}}{|x|^{s}}dx}, \qquad \mbox{ as } \|(\varphi,\psi)\|_{\mathbb{D}}\to 0,
\end{equation}
\begin{equation}\label{texpan2}
\overline{t}^{2^*_{s}}= 1-\frac{2^*_{s}}{2^*_{s}-2}\frac{\|\varphi \|^{2}_{\lambda_1}(1+o(1))}{\displaystyle\int_{\mathbb{R}^N}\dfrac{|z_\mu^{(2)}+\psi|^{2^*_{s}}}{|x|^{s}}dx}, \qquad \mbox{ as } \|(\varphi,\psi)\|_{\mathbb{D}}\to 0.
\end{equation}
As the decoupled energy functional $\mathcal{J}_2$ achieves its minimum in $z_\mu^{(2)}$, then
\begin{equation*}
\mathcal{J}_2(\overline{t}(z_\mu^{(2)}+\psi))-\mathcal{J}_2(z_\mu^{(2)})=\mathcal{J}_\nu(0,\overline{t}(z_\mu^{(2)}+\psi))-\mathcal{J}_\nu(0,z_\mu^{(2)}) \ge 0.
\end{equation*}
Next, comparing the energy of $(\varphi,z_\mu^{(2)}+\psi)$ and $(0,\overline{t}(z_\mu^{(2)}+\psi))$, we obtain
\begin{equation}\label{compar}
\begin{split}
 \mathcal{J}_\nu(\varphi,z_\mu^{(2)}+\psi)\!-\!\mathcal{J}_\nu(0,\overline{t}(z_\mu^{(2)}+\psi))\!=&\ \frac{1}{2} \| \varphi\|_{\lambda_1}^2+\frac{1}{2}(1-\overline{t}^2)\|z_\mu^{(2)}+\psi\|_{\lambda_2}^2\\
 &-\!\frac{1}{2_{s}^*}\! \int_{\mathbb{R}^N}\!\dfrac{| \varphi|^{2^*_{s}}}{|x|^{s}}dx+\!\frac{1}{2_{s}^*}(1-\overline{t}^{2^*_{s}})\!\!\! \int_{\mathbb{R}^N}\dfrac{|z_\mu^{(2)}+\psi|^{2^*_{s}}}{|x|^{s}}dx\\
& - \nu \int_{\mathbb{R}^N} h \frac{|\varphi|^{\alpha} |z_\mu^{(2)}+\psi|^{\beta}}{|x|^s}dx \\
=&\ \frac{1}{2} (1+o(1)) \| \varphi\|_{\lambda_1}^2,\qquad \mbox{ as } \|(\varphi,\psi)\|_{\mathbb{D}}\to 0.
\end{split}
\end{equation}
From two previous inequalities we get $\mathcal{J}_\nu(\varphi,z_\mu^{(2)}+\psi)-\mathcal{J}_\nu(0,z_\mu^{(2)}))\ge 0.$
Taking the perturbation $(\varphi,z_\mu^{(2)}+\psi)$ small enough in the $\mathbb{D}$-norm sense, we conclude that $(0,z_\mu^{(2)})$ is a local minimum of $\mathcal{J}_\nu$ on $\mathcal{N}_\nu$. In the case $\alpha=2$, from \eqref{compar} and \eqref{Hold2}, we get
\begin{equation*}
\begin{split}
\mathcal{J}_\nu(\varphi,z_\mu^{(2)}+\psi)-\mathcal{J}_\nu(0,\overline{t}(z_\mu^{(2)}+\psi)) &= \frac{1}{2} (1+o(1)) \| \varphi\|_{\lambda_1}^2 - \nu \int_{\mathbb{R}^N} h \frac{|\varphi|^{2} |z_\mu^{(2)}+\psi|^{\beta}}{|x|^s}dx \\
&= \left( \frac{1}{2}-\nu C'+o(1) \right) \, \| \varphi\|_{\lambda_1}^2,\qquad \mbox{ as } \|(\varphi,\psi)\|_{\mathbb{D}}\to 0.
\end{split}
\end{equation*}
Then for $\nu$ sufficiently small, we conclude that $(0,z_\mu^{(2)})$ is a local minimum of $\mathcal{J}_{\nu}$ in ${\mathcal{N}_\nu}$.

Next, we prove $iii)$ and $iv)$ follows similarly. Assume $\alpha<2$ and let $f(t)$ be the unique scalar satisfying $\displaystyle
(f(t)t\varphi,f(t)z_\mu^{(2)})\in\mathcal{N}_\nu$, with $\varphi\in \mathcal{D}^{1,2}(\mathbb{R}^N)\setminus\{0\}$ and $\mu>0$. Then, by \eqref{Nnueq1},
\begin{equation*}
\begin{split}
t^2\|\varphi\|^2_{\lambda_1}+\|z_\mu^{(2)}\|^2_{\lambda_2}
 =&[f(t)]^{2^*_{s}-2}|t|^{2^*_{s}}\left( \int_{\mathbb{R}^N}\dfrac{| \varphi|^{2^*_{s }}}{|x|^{s}}dx+\int_{\mathbb{R}^N}\dfrac{| z_\mu^{(2)}|^{2^*_{s }}}{|x|^{s}}dx \right)\\
& + (\alpha+\beta)\nu[f(t)]^{\alpha+\beta-2}|t|^{\alpha}\int_{\mathbb{R}^N} h \frac{|\varphi|^{\alpha} |z_\mu^{(2)}|^{\beta}}{|x|^s}dx.
\end{split}
\end{equation*}
Observe that $f(0)=1$. Moreover, $f \in C^1(\mathbb{R})$ by the implicit function theorem and
\begin{equation*}
f'(t)=\frac{\displaystyle2t\|\varphi\|^2_{\lambda_1} - 2^*_{s}f^{2^*_{s}-2}|t|^{2^*_{s}-2}t \int_{\mathbb{R}^N}\dfrac{| \varphi|^{2^*_{s}}}{|x|^{s}} - \alpha (\alpha+\beta)\nu f^{\alpha+\beta-2}|t|^{\alpha-2}t\int_{\mathbb{R}^N} h \frac{|\varphi|^{\alpha} |z_\mu^{(2)}|^{\beta}}{|x|^s} }{\displaystyle(2^*_{s}-2)f^{2^*_{s}-3}|t|^{2^*_{s}}\left(\!\int_{\mathbb{R}^N}\dfrac{| \varphi|^{2^*_{s}}}{|x|^{s}} + \int_{\mathbb{R}^N}\dfrac{| z_\mu^{(2)}|^{2^*_{s}}}{|x|^{s}}\right) + \nu \delta f^{\alpha+\beta-3}|t|^{\alpha}\int_{\mathbb{R}^N} h \frac{|\varphi|^{\alpha} |z_\mu^{(2)}|^{\beta}}{|x|^s}},
\end{equation*}
where $\delta=(\alpha+\beta)(\alpha+\beta-2)$. Since $\alpha<2$, we can write the previous expression as 
\begin{equation*}
f'(t)=\dfrac{\displaystyle- \alpha (\alpha+\beta)\nu\int_{\mathbb{R}^N} h \frac{|\varphi|^{\alpha} |z_\mu^{(2)}|^{\beta}}{|x|^s}dx }{\displaystyle (2^*_{s}-2) \int_{\mathbb{R}^N}\dfrac{| z_\mu^{(2)}|^{2^*_{s}}}{|x|^{s}}dx }\, |t|^{\alpha-2}t \, \, (1+o(1)), \qquad \mbox{ as }  t\to 0.
\end{equation*}
By integrating and using that $f(0)=1$, we obtain
\begin{equation*}
f(t)= 1- \frac{\displaystyle (\alpha+\beta)\nu\int_{\mathbb{R}^N} h \frac{|\varphi|^{\alpha} |z_\mu^{(2)}|^{\beta}}{|x|^s}dx }{\displaystyle (2^*_{s}-2) \int_{\mathbb{R}^N}\frac{| z_\mu^{(2)}|^{2^*_{s}}}{|x|^{s}} dx}\, |t|^{\alpha} \, \, (1+o(1)), \qquad \mbox{ as }  t\to 0,
\end{equation*}
and consequently, by Taylor expansion,
\begin{equation}\label{efe2star}
f^{2^*_{s}}(t)= 1- \frac{\displaystyle (N-s)(\alpha+\beta)\nu\int_{\mathbb{R}^N} h\frac{|\varphi|^{\alpha} |z_\mu^{(2)}|^{\beta}}{|x|^s} dx}{\displaystyle (2-s) \int_{\mathbb{R}^N}\dfrac{| z_\mu^{(2)}|^{2^*_{s}}}{|x|^{s}} dx}\, |t|^{\alpha} \, \, (1+o(1)), \qquad \mbox{ as }  t\to 0.
\end{equation}
By comparing the energy of the couples $(f(t)t\varphi,f(t)z_\mu^{(2)})$ and $(0,z_\mu^{(2)})$ using \eqref{Nnueq2} and \eqref{efe2star},
\begin{equation*}\label{compar2}
\begin{split}
\mathcal{J}_\nu&((f(t)t\varphi,f(t)z_\mu^{(2)})-\mathcal{J}_\nu(0,z_\mu^{(2)})\\ 
&=\frac{2-s}{2(N-s)}\left([f(t)]^{2^*_{s}}\int_{\mathbb{R}^N}\dfrac{| \varphi|^{2^*_{s}}}{|x|^{s}}dx + ([f(t)]^{2^*_{s}}-1)  \int_{\mathbb{R}^N}\dfrac{|z_\mu^{(2)}|^{2^*_{s}}}{|x|^{s}}dx\right)\\
&\ + \left(\frac{\alpha+\beta-2}{2}\right)\nu |t|^{\alpha} \int_{\mathbb{R}^N} h \frac{|\varphi|^{\alpha} |z_\mu^{(2)}|^{\beta}}{|x|^s}dx \\
&=- \left(\frac{\alpha+\beta}{2}\right)\nu |t|^{\alpha}\!\! \int_{\mathbb{R}^N} h \frac{|\varphi|^{\alpha} |z_\mu^{(2)}|^{\beta}}{|x|^s}dx + \left(\frac{\alpha+\beta-2}{2}\right)\nu |t|^{\alpha}\!\! \int_{\mathbb{R}^N} h \frac{|\varphi|^{\alpha} |z_\mu^{(2)}|^{\beta}}{|x|^s}dx+o(|t|^\alpha) \\
&=- \nu |t|^{\alpha} \int_{\mathbb{R}^N} h \frac{|\varphi|^{\alpha} |z_\mu^{(2)}|^{\beta}}{|x|^s}dx+o(|t|^\alpha)\quad\text{as }t\to0.
\end{split}
\end{equation*}
This implies that $\mathcal{J}_\nu((f(t)t\varphi,f(t)z_\mu^{(2)})-\mathcal{J}_\nu(0,z_\mu^{(2)}))<0$ for $t$ sufficiently small, so $(0,z_\mu^{(2)})$ is a local minimum of $\mathcal{J}_\nu$ in $\mathcal{N}_\nu$ for the previous path. Since $z_\mu^{(2)}$ is an absolute minimum of the decouple functional $\mathcal{J}_2$, for any couple $(0,\psi) \in \mathcal{N}_\nu$, we have $\mathcal{J}_\nu(0,\psi)>\mathcal{J}_\nu(0,z_\mu^{(2)})$, so $(0,z_\mu^{(2)})$ is a local minimum in $\{0\}\times \mathcal{N}_2\subset\mathcal{N}_{\nu}$. Thus, if $\alpha<2$, the couple $(0,z_\mu^{(2)})$ is a saddle point of $\mathcal{J}_\nu$ in $\mathcal{N}_\nu$. If $\alpha=2$, by considering again the couple $((f(t)t\varphi,f(t)z_\mu^{(2)}) \in \mathcal{N}_\nu$, we get
\begin{equation*}
f'(t)=\frac{\displaystyle 2\|\varphi\|^2_{\lambda_1}- 2 (2+\beta)\nu\int_{\mathbb{R}^N} h \frac{|\varphi|^{2} |z_\mu^{(2)}|^{\beta}}{|x|^s}dx}{\displaystyle (2^*_{s}-2) \int_{\mathbb{R}^N}\dfrac{| z_\mu^{(2)}|^{2^*_{s}}}{|x|^{s}}dx }\, t \, \, (1+o(1)), \qquad \mbox{ as }  t\to 0,
\end{equation*}
so that
\begin{equation*}
f(t)=1-\frac{\displaystyle (2+\beta)\nu\int_{\mathbb{R}^N} h \frac{|\varphi|^{2} |z_\mu^{(2)}|^{\beta}}{|x|^s}dx -\|\varphi\|^2_{\lambda_1} }{\displaystyle (2^*_{s}-2) \int_{\mathbb{R}^N}\dfrac{| z_\mu^{(2)}|^{2^*_{s}}}{|x|^{s}} dx}\, t^2 \, \, (1+o(1)), \qquad \mbox{ as }  t\to 0,
\end{equation*}
and we conclude
\begin{equation*}
\mathcal{J}_\nu((f(t)t\varphi,f(t)z_\mu^{(2)})-\mathcal{J}_\nu(0,z_\mu^{(2)}) = \left(\frac{1}{2} \|\varphi\|_{\lambda_1}^2 - \nu \int_{\mathbb{R}^N} h \frac{|\varphi|^{2} |z_\mu^{(2)}|^{\beta}}{|x|^s}dx \right)t^2 +o(t^2).
\end{equation*}
Subsequently, if $\nu$ is large enough the previous quantity will be negative and $(0,z_\mu^{(2)})$ is a saddle point of $\mathcal{J}_\nu$ in $\mathcal{N}_\nu$. Therefore, reasoning as before, the thesis $iii)$ is concluded.
\end{proof}
Before ending this preliminary section, we state the following algebraic result which extends \cite[Lemma 3.3]{AbFePe}, corresponding to $s=0$, to our weighted setting dealing with $s\in(0,2)$. Its proof follows analogously so we omit the details.
\begin{lemma}\label{algelemma}
Let $A, B>0$, $0\leq s<2$ and $\theta \ge 2$ and consider the set
\begin{equation*}
\Sigma_\nu=\{\sigma \in (0,+\infty)  \, : \, A \sigma^{\frac{2}{2_{s}^*}} < \sigma + B \nu \sigma^{\frac{\theta}{2_s^*}} \}.
\end{equation*}
Then, for every $\varepsilon>0$ there exists $\tilde{\nu}>0$ such that
\begin{equation*}
\inf_{\Sigma_\nu} \sigma > (1-\varepsilon) A^{\frac{N-s}{2-s}} \qquad \mbox{ for any } 0<\nu<\tilde{\nu}.
\end{equation*}
\end{lemma}
\section{The Palais-Smale condition}\label{section:PS}
A crucial step to obtain the existence of solutions relies on the compactness of the energy functional $\mathcal{J}_\nu$ provided by the PS condition. 
\begin{definition}
Let $V$ be a Banach space. We say that $\{u_n\} \subset V$ is a PS sequence at level $c$ for an energy functional $\mathfrak{F}:V\mapsto\mathbb{R}$ if
\begin{equation*}
\mathfrak{F}(u_n) \to c \quad \hbox{ and }\quad  \mathfrak{F}'(u_n) \to 0\quad\mbox{in}\ V'\quad \hbox{as}\quad n\to + \infty,
\end{equation*}
where $V'$ is the dual space of $V$. Moreover, we say that the functional $\mathfrak{F}$ satisfies the PS condition at level $c$ if every PS sequence at $c$ for $\mathfrak{F}$ has a strongly convergent subsequence.
\end{definition}
Next we state some results useful in the sequel. Their proofs are somehow standard and similar to \cite[Lemma 3.2]{CoLSOr} and \cite[Lemma 3.3]{CoLSOr} respectively, so we omit the details.
\begin{lemma}\label{lemma:PSNehari}
Let $\{(u_n,v_n)\} \subset \mathcal{N}_\nu$ be a PS sequence for $\mathcal{J}_\nu {\big|}_{\mathcal{N}_\nu}$ at level $c\in\mathbb{R}$. Then, $\{(u_n,v_n)\}$ is a PS sequence for $\mathcal{J}_\nu$ in $\mathbb{D}$, namely
\begin{equation}\label{PSNehari}
\mathcal{J}_{\nu}'(u_n,v_n)\to0 \quad \mbox{ in } \mathbb{D}' \quad\text{as }n\to+\infty.
\end{equation}
\end{lemma}
\begin{lemma}\label{lemmaPS0}
Let $\{(u_n,v_n)\} \subset \mathbb{D}$ be a PS sequence for $\mathcal{J}_\nu$ at level $c\in\mathbb{R}$. Then,  $\|(u_n,v_n)\|_{\mathbb{D}}<C$.
\end{lemma}
\subsection{Subcritical range $ \alpha+\beta < 2_{s}^*$}\hfill\newline
We establish now the Palais--Smale condition for subcritical energy levels of $\mathcal{J}_\nu$, which will allow us to find existence of solutions for \eqref{system:alphabeta} by minimization. 
\begin{lemma}\label{lemmaPS2}
 Assume $\alpha+\beta<2_{s}^*$.  Then, the functional $\mathcal{J}_\nu$ satisfies the PS condition for any level $c$ such that
\begin{equation}\label{hyplemmaPS2}
c<\min\left\{ \mathfrak{C}(\lambda_1,s), \mathfrak{C}(\lambda_2,s) \right\}.
\end{equation}
\end{lemma}
\begin{proof}
By Lemma~\ref{lemmaPS0}, any PS sequence is bounded in $\mathbb{D}$. Thus, there exists $(\tilde{u},\tilde{v})\in\mathbb{D}$ and a subsequence (denoted also by $\{(u_n,v_n)\}$) such that 
 \begin{align*}
(u_n,v_n) \rightharpoonup (\tilde{u},\tilde{v})& \quad \hbox{weakly in  } \mathbb{D},\\
(u_n,v_n) \to (\tilde{u},\tilde{v})&\quad \hbox{strongly in  } L^q(\mathbb{R}^N)\times L^q(\mathbb{R}^N)\text{ for } 1\leq q<2_s^*,\\
(u_n,v_n) \to (\tilde{u},\tilde{v})&\quad \hbox{a.e. in  }\mathbb{R}^N.
\end{align*}
By the \textit{concentration-compactness principle} (cf. \cite{Lions1,Lions2}), there exist a subsequence (still denoted by) $\{(u_n,v_n)\}$ and positive numbers $\mu_0$, $\rho_0$, $\eta_0$, $\overline{\mu}_0$, $\overline{\rho}_0$ and $\overline{\eta}_0$ such that, in the sense of measures,
\begin{equation}\label{con-comp}
\left\{
\begin{array}{l}
|\nabla u_n|^2 \rightharpoonup d\mu\ge |\nabla \tilde{u}|^2+\mu_0\delta_0,\qquad |\nabla v_n|^2\rightharpoonup d\overline{\mu}\ge |\nabla \tilde{v}|^2+\overline{\mu}_0\delta_0,\\
\\
\dfrac{|u_n|^{2_{s}^*}}{|x|^{s}} \rightharpoonup d\rho=\dfrac{|\tilde{u}|^{2_{s}^*}}{|x|^{s}}+\rho_0\delta_0,\qquad\ \ 
\dfrac{|v_n|^{2_{s}^*}}{|x|^{s}}  \rightharpoonup d\overline{\rho}= \dfrac{|\tilde{v}|^{2_{s}^*}}{|x|^{s}}+\overline{\rho}_0\delta_0,\\
\\
\dfrac{u_n^2}{|x|^2} \rightharpoonup d\eta=\dfrac{\tilde{u}^2}{|x|^2}+\eta_0\delta_0,\qquad \mkern+40mu 
\dfrac{v_n^2}{|x|^2} \rightharpoonup d\overline{\eta}=\dfrac{\tilde{v}^2}{|x|^2}+\overline{\eta}_0\delta_0.
\end{array}
\right.
\end{equation}
Because of \eqref{hardy_inequality} and \eqref{hardy_sobolev_inequality}, the above numbers satisfy the inequalities
\begin{equation}\label{ineq:harcon}
\Lambda_N\eta_0\leq\mu_0\qquad\text{and}\qquad\Lambda_N\overline{\eta}_0\leq\overline{\mu}_0.
\end{equation}
The concentration at infinity of the sequence $\{u_n\}$ is encoded by the numbers
\begin{equation}\label{con-compinfty}
\begin{split}
\mu_{\infty}&=\lim\limits_{R\to+\infty}\limsup\limits_{n\to+\infty}\int_{|x|>R}|\nabla u_n|^{2}dx,\\
\rho_{\infty}&=\lim\limits_{R\to+\infty}\limsup\limits_{n\to+\infty}\int_{|x|>R}\frac{|u_n|^{2_{s}^*}}{|x|^{s}}dx,\\
\eta_{\infty}&=\lim\limits_{R\to+\infty}\limsup\limits_{n\to+\infty}\int_{|x|>R}\frac{u_n^{2}}{|x|^2}dx.
\end{split}
\end{equation}
The concentration at infinity of $\{v_n\}$ is encoded by the numbers $\overline{\mu}_{\infty}$, $\overline{\rho}_{\infty}$ and $\overline{\eta}_{\infty}$ defined analogously. Let $\varphi_{\varepsilon}(x)$ be a smooth cut-off function centered at $0$, i.e., $\varphi_{\varepsilon}\in C^{\infty}(\mathbb{R}^+_0)$,
\begin{equation}\label{cutoff}
\varphi_{\varepsilon}=1 \quad \hbox{in}\quad B_{\frac{\varepsilon}{2}}(0),\quad \varphi_{\varepsilon}=0 \quad \hbox{in}
\quad B_{\varepsilon}^c(0)\quad \hbox{and}\quad\displaystyle|\nabla \varphi_{\varepsilon}|\leq \frac{4}{\varepsilon},
\end{equation}
where $B_r(0)$ is the ball of radius $r>0$ centered at $0$. Testing $\mathcal{J}_{\nu}'(u_n,v_n)$ with $(u_n\varphi_{\varepsilon},0)$ (resp. with $(0,v_n\varphi_{0,\varepsilon})$) we get
\begin{equation}\label{eq:concompact1}
\begin{split}
0&=\lim\limits_{n\to+\infty}\left(\int_{\mathbb{R}^N}|\nabla u_n|^2\varphi_{\varepsilon}dx+\int_{\mathbb{R}^N}u_n\nabla u_n\nabla\varphi_{\varepsilon}dx-\lambda_1\int_{\mathbb{R}^N}\frac{u_n^2}{|x|^2}\varphi_{\varepsilon}dx\right.\\
&\mkern+80mu-\left.\int_{\mathbb{R}^N}\frac{|u_n|^{2_{s}^*}}{|x|^{s}}\varphi_{\varepsilon}dx-\nu\alpha\int_{\mathbb{R}^N}h\frac{|u_n|^\alpha  |v_n|^\beta}{|x|^{s}}  \varphi_{\varepsilon} dx\right)\\
&=\int_{\mathbb{R}^N}\varphi_{\varepsilon}d\mu+\int_{\mathbb{R}^N}\tilde{u}\nabla \tilde{u}\nabla\varphi_{\varepsilon}dx-\lambda_1\int_{\mathbb{R}^N}\varphi_{\varepsilon}d\eta-\int_{\mathbb{R}^N}\varphi_{\varepsilon}d\rho-\nu\alpha\int_{\mathbb{R}^N}h\frac{|\tilde{u}|^\alpha |\tilde{v}|^{\beta}}{|x|^{s}} \varphi_{\varepsilon} \, dx.
\end{split}
\end{equation}
Taking $\varepsilon\to0$ we find $\mu_0-\lambda_1\eta_0-\rho_0\leq 0$ (resp. $\overline{\mu}_0-\lambda_2\overline{\eta}_0-\overline{\rho}_0\leq0$).
By \eqref{H-S_lambda}, we also have
\begin{equation}\label{ineq:con0_a}
\mu_0-\lambda_1\eta_0\ge \mathcal{S}(\lambda_1,s)\rho_0^{\frac{2}{2_{s}^*}}\qquad\text{and}\qquad
\overline{\mu}_0-\lambda_2\overline{\eta}_0\ge \mathcal{S}(\lambda_2,s)\overline{\rho}_0^{\frac{2}{2_{s}^*}},
\end{equation}
from where we conclude
\begin{equation}\label{ineq:con0}
\begin{split}
\rho_0&=0\quad\text{or}\quad \rho_0\ge \left[\mathcal{S}(\lambda_1,s)\right]^{\frac{N-s}{2-s}},\\
\overline{\rho}_0&=0\quad\text{or}\quad \overline{\rho}_0\ge \left[\mathcal{S}(\lambda_2,s)\right]^{\frac{N-s}{2-s}}.
\end{split}
\end{equation}
Finally, for $R>0$, consider $\varphi_{\infty,\varepsilon}$ a cut-off function supported near $\infty$, i.e.,
\begin{equation}\label{cutoffinfi}
\varphi_{\infty,\varepsilon}=0 \quad \hbox{in}\quad B_{R}(0),\quad \varphi_{\infty,\varepsilon}=1 \quad \hbox{in}
\quad B_{R+1}^c(0)\quad \hbox{and}\quad\displaystyle|\nabla \varphi_{\infty,\varepsilon}|\leq \frac{4}{\varepsilon}.
\end{equation}
Testing $\mathcal{J}_{\nu}'(u_n,v_n)$ with $(u_n\varphi_{\infty,\varepsilon},0)$, we can similarly prove that $\mu_{\infty}-\lambda_1\eta_{\infty}-\rho_{\infty}\leq 0$. Moreover, testing $\mathcal{J}_{\nu}'(u_n,v_n)$ with $(0,v_n\varphi_{\infty,\varepsilon})$ we also get $\overline{\mu}_{\infty}-\lambda_2\overline{\eta}_{\infty}-\overline{\rho}_{\infty}\leq0$. Thus
\begin{equation}\label{ineq:coninf_a}
\mu_{\infty}-\lambda_1\eta_{\infty}\ge \mathcal{S}(\lambda_1,s)\rho_{\infty}^{\frac{2}{2_{s}^*}}\qquad\text{and}\qquad
\overline{\mu}_{\infty}-\lambda_2\overline{\eta}_{\infty}\ge \mathcal{S}(\lambda_2,s)\overline{\rho}_{\infty}^{\frac{2}{2_{s}^*}},
\end{equation}
and we also conclude
\begin{equation}\label{ineq:coninf}
\begin{split}
\rho_{\infty}&=0\quad\text{or}\quad \rho_{\infty}\ge \left[\mathcal{S}(\lambda_1,s)\right]^{\frac{N-s}{2-s}},\\
\overline{\rho}_{\infty}&=0\quad\text{or}\quad \overline{\rho}_{\infty}\ge \left[\mathcal{S}(\lambda_2,s)\right]^{\frac{N-s}{2-s}}.
\end{split}
\end{equation}
Next, since given $\{(u_n,v_n)\}$ a PS sequence for $\mathcal{J}_{\nu}$ at level $c$, we have $\mathcal{J}_{\nu}(u_n,v_n)\to c$ and $\mathcal{J}_{\nu}'(u_n,v_n)\to0$ as $n\to+\infty$, it follows that, up to a subsequence (still denoted by $\{(u_n,v_n)\}$),
\begin{equation*}
\mathcal{J}_{\nu}(u_n,v_n)-\frac{1}{\alpha+\beta}\left\langle \mathcal{J}_{\nu}'(u_n,v_n)\left|\frac{(u_n,v_n)}{\|(u_n,v_n)\|_{\mathbb{D}}} \right.\right\rangle=c+\|(u_n,v_n)\|_{\mathbb{D}}\cdot o(1),
\end{equation*}
that is
\begin{equation}\label{eq:limit2}
\begin{split}
c=&\left( \frac{1}{2}- \frac{1}{\alpha+\beta} \right)\|(u_n,v_n)\|_{\mathbb{D}}^2\\
&+\left( \frac{1}{\alpha+\beta}-\frac{1}{2^*_{s}} \right)\left(\int_{\mathbb{R}^N} \frac{|u_n|^{2^*_{s}}}{|x|^{s}}dx +\int_{\mathbb{R}^N} \frac{|v_n|^{2^*_{s}}}{|x|^{s}}dx\right)+o(1).
\end{split}
\end{equation}
Hence, because of \eqref{con-comp}, \eqref{ineq:con0_a} and \eqref{ineq:coninf_a} above, we find
\begin{equation}\label{ineq:larga}
\begin{split}
c\ge& \left(\frac{1}{2}-\frac{1}{\alpha+\beta} \right)\Bigg(\|(\tilde{u},\tilde{v})\|_{\mathbb{D}}^2 + (\mu_0-\lambda_1\eta_0)+(\mu_{\infty}-\lambda_1\eta_{\infty}) \\
&\mkern+199mu +(\overline{\mu}_0-\lambda_2\overline{\eta}_0)+(\overline{\mu}_{\infty}-\lambda_2\overline{\eta}_{\infty})\Bigg)\\
&+\left(\frac{1}{\alpha+\beta}- \frac{1}{2_{s}^*} \right)\left(\int_{\mathbb{R}^N}\frac{|\tilde{u}|^{2_{s}^*}}{|x|^{s}}dx+  \rho_0+\rho_{\infty} +\int_{\mathbb{R}^N}\frac{|\tilde{v}|^{2_{s}^*}}{|x|^{s}}dx +\overline{\rho}_0+\overline{\rho}_{\infty} \right)\\
\ge&\left(\frac{1}{2}-\frac{1}{\alpha+\beta} \right)\left( \mathcal{S}(\lambda_1,s)\left[\rho_0^{\frac{2}{2_{s}^*}}+\rho_{\infty}^{\frac{2}{2_{s}^*}}\right]+\mathcal{S}(\lambda_2,s)\left[\overline{\rho}_0^{\frac{2}{2_{s}^*}}+\overline{\rho}_{\infty}^{\frac{2}{2_{s}^*}}\right]\right)\\
&+\left(\frac{1}{\alpha+\beta}- \frac{1}{2_{s}^*} \right)\left(\rho_0+\rho_{\infty} +\overline{\rho}_0+\overline{\rho}_{\infty} \right).
\end{split}
\end{equation}
If $\rho_0\neq0$, from \eqref{ineq:larga} and \eqref{ineq:con0}, we get
\begin{equation*}
c\ge\frac{2-s}{2(N-s)}\left[\mathcal{S}(\lambda_1,s)\right]^{\frac{N-s}{2-s}}=\mathfrak{C}(\lambda_1,s),
\end{equation*}
and we reach a contradiction with the hypothesis on the energy level $c$. Then, $\rho_0=0$. Similarly, we also get $\overline{\rho}_0=0$. Arguing as above and using \eqref{ineq:coninf} we also find $\rho_{\infty}=0$ and $\overline{\rho}_{\infty}=0$. Thus, there exists a subsequence strongly converging in $L^{2_s^*}(\mathbb{R}^N)\times L^{2_s^*}(\mathbb{R}^N)$, so that
\begin{equation*}
\|(u_n-\tilde{u},v_n-\tilde{v})\|_{\mathbb{D}}^2 = \left\langle \mathcal{J}_\nu'(u_n,v_n)\big| (u_n-\tilde{u},v_n-\tilde{v}) \right\rangle + o(1).
\end{equation*}
Therefore, the sequence $\{(u_n,v_n)\}$ strongly converges in $\mathbb{D}$ and the PS condition holds.
\end{proof}

Next, we improve Lemma~\ref{lemmaPS2} by proving the PS condition for supercritical energy levels, excluding multiples or combinations of the critical ones. 

In order to find positive solutions of \eqref{system:alphabeta} we consider the truncated problem
\begin{equation}\label{systemp}
\left\{\begin{array}{ll}
\displaystyle -\Delta u - \lambda_1 \frac{u}{|x|^2}-\frac{(u^+)^{2_{s}^*-1}}{|x|^{s}}=  \nu\alpha h(x) \frac{(u^+)^{\alpha-1}\, (v^+)^{\beta}}{|x|^{s}}  &\text{in }\mathbb{R}^N,\vspace{.3cm}\\
\displaystyle -\Delta v - \lambda_2 \frac{v}{|x|^2}-\frac{(v^+)^{2_{s}^*-1}}{|x|^{s}}= \nu \beta h(x)  \frac{(u^+)^{\alpha}\, (v^+)^{\beta-1}}{|x|^{s}} &\text{in }\mathbb{R}^N,
\end{array}\right.
\end{equation}
where  $u^+=\max\{u,0\}$. Note that $u=u^+ + u^-$ where $u^-$ is negative part of the function $u$, i.e., $u^-=\min\{u,0\}$. Let us also note that a solution $(u,v)$ of \eqref{system:alphabeta} also satisfies \eqref{systemp}.

The system \eqref{systemp} is a variational system and its solutions correspond to critical points of
\begin{equation}\label{funct:SKdVp}
\mathcal{J}^+_\nu (u,v)= \|(u,v)\|^2_{\mathbb{D}}- \frac{1}{2^*_{s}} \int_{\mathbb{R}^N} \frac{(u^+)^{2^*_{s}}}{|x|^{s}} dx - \frac{1}{2^*_{s}} \int_{\mathbb{R}^N} \frac{(v^+)^{2^*_{s}}}{|x|^{s}} dx -\nu \int_{\mathbb{R}^N} h \frac{(u^+)^\alpha (v^+)^{\beta}}{|x|^{s}} dx,
\end{equation}
defined in $\mathbb{D}$. We will denote by $\mathcal{N}^+_\nu$ the Nehari manifold associated to $\mathcal{J}^+_\nu $. In particular,
\begin{equation*}
\mathcal{N}^+_\nu=\left\{ (u,v) \in \mathbb{D} \setminus \{(0,0)\} \, : \,  \left\langle (\mathcal{J}_\nu^+)' (u,v){\big|}(u,v)\right\rangle=0 \right\}.
\end{equation*}
Given $(u,v) \in \mathcal{N}^+_\nu$, the following identity holds
\begin{equation} \label{Nnueqp}
\|(u,v)\|_\mathbb{D}^2=\int_{\mathbb{R}^N} \frac{(u^+)^{2^*_{s}}}{|x|^{s}} dx + \int_{\mathbb{R}^N} \frac{(v^+)^{2^*_{s}}}{|x|^{s}} dx +\nu (\alpha+\beta) \int_{\mathbb{R}^N} h(x) \frac{(u^+)^{\alpha} (v^+)^{\beta}}{|x|^s}  dx.
\end{equation}
Observe that, on the Nehari manifold $\mathcal{N}^+_\nu$, the functional $\mathcal{J}^+_{\nu}$ read as
\begin{equation}\label{Nnueqp1}
\begin{split}
\mathcal{J}^+_{\nu}{\big|}_{\mathcal{N}_\nu} (u,v) =& \left( \frac{1}{2}-\frac{1}{\alpha+\beta} \right) \|(u,v)\|^2_{\mathbb{D}}\\
& + \left( \frac{1}{\alpha+\beta}-\frac{1}{2^*_{s}} \right)\left(\int_{\mathbb{R}^N} \frac{(u^+)^{2^*_{s}}}{|x|^{s}}dx+ \int_{\mathbb{R}^N} \frac{(v^+)^{2^*_{s}}}{|x|^{s}} dx\right).
\end{split}
\end{equation}
\begin{lemma}\label{lemmaPS1}
Assume that $\alpha+\beta<2_{s}^*$, $\alpha\ge2$ and $\lambda_1\leq\lambda_2$. Then, there exists $\tilde{\nu}>0$ such that, if $0<\nu\leq\tilde{\nu}$ and $\{(u_n,v_n)\} \subset \mathbb{D}$ is a PS sequence for $\mathcal{J}^+_\nu$ at level $c\in\mathbb{R}$ such that
\begin{equation}\label{PS1}
\mathfrak{C}(\lambda_1,s)<c<\mathfrak{C}(\lambda_1,s)+\mathfrak{C}(\lambda_2,s)
\end{equation}
and
\begin{equation}\label{PS2}
c\neq \ell \mathfrak{C}(\lambda_2,s) \quad \mbox{ for every } \ell \in \mathbb{N}\setminus \{0\},
\end{equation}
then $(u_n,v_n)\to(\tilde{u},\tilde{v}) \in \mathbb{D}$ up to subsequence.
\end{lemma}
\begin{proof}
Arguing as in Lemma~\ref{lemmaPS0}, any PS sequence for $\mathcal{J}_\nu^+$ is bounded in $\mathbb{D}$. Thus,  there exists a subsequence $\{(u_n,v_n)\}\rightharpoonup(\tilde{u},\tilde{v}) \in \mathbb{D}$. Since $(\mathcal{J}^+_\nu)'(u_n,v_n)\to 0$ in $\mathbb{D}'$, then
\begin{equation*}
\left\langle (\mathcal{J}^+_\nu)'(u_n,v_n){\big|} (u_n^-,0)\right\rangle= \int_{\mathbb{R}^N} |\nabla u_n^-|^2 \, dx - \lambda_1 \int_{\mathbb{R}^N} \dfrac{(u_n^-)^2}{|x|^2} \, dx  \to0,
\end{equation*}
and, hence, $u_n^-\to 0$ strongly in $\mathcal{D}^{1,2} (\mathbb{R}^N)$. Similarly, we can also prove $v_n^-\to 0$.
Thus, we can assume that $\{(u_n,v_n)\}$ is a non-negative PS sequence at level $c$ for $\mathcal{J}_\nu$.

As in the proof of Lemma~\ref{lemmaPS2}, we deduce the existence of a subsequence, still denoted by $\{(u_n,v_n)\}$ and positive numbers $\mu_0$, $\rho_0$, $\eta_0$, $\overline{\mu}_0$, $\overline{\rho}_0$ and $\overline{\eta}_0$ such that \eqref{con-comp} holds. In addition, the inequalities \eqref{ineq:con0_a}, \eqref{ineq:con0} also hold. Analogously, the concentration at infinity is encoded by the values $\mu_\infty$, $\rho_\infty$, $\overline{\mu}_\infty$ and $\overline{\rho}_\infty$ as in \eqref{con-compinfty}, for which \eqref{ineq:coninf_a} and \eqref{ineq:coninf} also hold.

Next, we claim:
\begin{equation}\label{claimPS21}
\mbox{ either } u_n\to \tilde u\  \mbox{  strongly in } L^{2_s^*}(\mathbb{R}^N) \qquad \mbox{or} \qquad v_n\to \tilde v\ \mbox{  strongly in }L^{2_s^*}(\mathbb{R}^N).
\end{equation}

Let us argue by contradiction. Assume that $\{u_n\}$ and $\{v_n\}$ do not strongly converge in $L^{2_s^*}(\mathbb{R}^N)$. Then, there exists $j,k\in \{0, \infty\}$ such that $\rho_{j}>0$ and $\overline{\rho}_k>0$. Thus, as in \eqref{eq:limit2}, because of \eqref{con-comp}, \eqref{ineq:con0_a}, \eqref{ineq:con0}, \eqref{ineq:coninf_a} and \eqref{ineq:coninf} applied in \eqref{ineq:larga} we get
\begin{equation*}
\begin{split}
c=&\left( \frac{1}{2} -\frac{1}{\alpha+\beta} \right)\|(u_n,v_n)\|_{\mathbb{D}}^2+\left( \frac{1}{\alpha+\beta}-\frac{1}{2^*_{s}} \right)\left(\int_{\mathbb{R}^N} \frac{|u_n|^{2^*_{s}}}{|x|^{s}}dx + \int_{\mathbb{R}^N} \frac{|v_n|^{2^*_{s}}}{|x|^{s}}dx\right)+o(1)\\
\ge&\left( \frac{1}{2} -\frac{1}{\alpha+\beta} \right)\left(\mathcal{S}(\lambda_1,s)\rho_j^{\frac{2}{2_{s}^*}} +\mathcal{S}(\lambda_2,s) \overline{\rho}_k^{\frac{2}{2_{s}^*}}\right)+\left( \frac{1}{\alpha+\beta}-\frac{1}{2^*_{s}} \right)(\rho_j+\overline{\rho}_k) \\
\ge&\, \frac{2-s}{2(N-s)}\left(\left[\mathcal{S}(\lambda_1,s)\right]^{\frac{N-s}{2-s}}+\left[\mathcal{S}(\lambda_2,s)\right]^{\frac{N-s}{2-s}}\right)\\
=&\,\mathfrak{C}(\lambda_1,s)+\mathfrak{C}(\lambda_2,s),
\end{split}
\end{equation*}
in contradiction with \eqref{PS1}, so \eqref{claimPS21} is proved. Consequently, we claim:
\begin{equation}\label{claimPS22}
\mbox{ either } u_n\to  \tilde u\  \mbox{ in } \mathcal{D}^{1,2}(\mathbb{R}^N) \qquad \mbox{ or } \qquad v_n\to  \tilde v \ \mbox{ in } \mathcal{D}^{1,2}(\mathbb{R}^N).
\end{equation}
By \eqref{claimPS21}, we can assume that the sequence $\{u_n\}$ strongly converges in $L^{2_s^*}(\mathbb{R}^N)$. Then, since
\begin{equation*}
\|u_n-\tilde u\|_{\lambda_1}^2 = \left\langle \mathcal{J}_\nu'(u_n,v_n){\big|} (u_n-\tilde u,0) \right\rangle + o(1),
\end{equation*}
it follows that $u_n\to \tilde u$ in $\mathcal{D}^{1,2}(\mathbb{R}^N)$. Repeating the argument for $\{v_n\}$ we conclude \eqref{claimPS22}. 

Next, we prove both $\{u_n\}$ and $\{v_n\}$ strongly converge in $\mathcal{D}^{1,2}(\mathbb{R}^N)$. 

\textbf{Case 1}: The sequence $\{v_n\}$ strongly converges to $\tilde{v}$ in $\mathcal{D}^{1,2}(\mathbb{R}^N)$.\hfill\newline 
Let us prove that $\{u_n\}$ strongly converges to $\tilde{u}$ in $\mathcal{D}^{1,2}(\mathbb{R}^N)$. Assume, by contradiction, that none of its subsequences converge. If there is concentration at $0$ and $\infty$, by of \eqref{ineq:con0_a}, \eqref{ineq:con0}, \eqref{ineq:coninf_a}, \eqref{ineq:coninf} and \eqref{ineq:larga}, we get
\begin{equation*}
c\ge \frac{2-s}{N-s}\left[\mathcal{S}(\lambda_1,s)\right]^{\frac{N-s}{2-s}}=2\mathfrak{C}(\lambda_1,s)\ge \mathfrak{C}(\lambda_1,s)+\mathfrak{C}(\lambda_2,s).
\end{equation*}
This contradicts \eqref{PS1}, so the sequence $\{u_n\}$ concentrates only at $0$ or at $\infty$. Next, we prove $\tilde{v}\not \equiv 0$. Assume by contradiction that $\tilde{v}\equiv 0$, then $\tilde{u}\ge 0 $ and $\tilde{u}$ verifies \eqref{entire} with $j=1$. Thus, $\tilde{u}=z_\mu^{(1)}$ for some $\mu>0$ and, by \eqref{normcrit}, also  $\displaystyle\int_{\R^N}\frac{ \tilde{u}^{2_{s}^*}}{|x|^{s}} dx =[ \mathcal{S}(\lambda_1,s)]^{\frac{N-s}{2-s}}$. As $\{u_n\}$ concentrates at one point, by \eqref{ineq:larga} jointly with \eqref{ineq:con0_a}, \eqref{ineq:con0} and $\lambda_1\leq\lambda_2$ we conclude
\begin{equation}\label{eq:0}
\begin{split}
c&\ge \frac{2-s}{2(N-s)}\left( \int_{\R^N}\frac{ \tilde{u}^{2_{s}^*}}{|x|^{s}} dx +\left[\mathcal{S}(\lambda_1,s)\right]^{\frac{N-s}{2-s}}  \right)=2\mathfrak{C}(\lambda_1,s) \ge \mathfrak{C}(\lambda_1,s)+\mathfrak{C}(\lambda_2,s),
\end{split}
\end{equation}
in contradiction with \eqref{PS1}. If $\tilde{v}\equiv 0$ and $\tilde{u}\equiv 0$, then $u_n$ satisfies 
\begin{equation*}
-\Delta u_n - \lambda_1 \frac{u_n}{|x|^2}-\frac{u_n^{2_{s}^*-1}}{|x|^{s}}=o(1) \qquad \mbox{ in the dual space } \left( \mathcal{D}^{1,2} (\mathbb{R}^N)\right)',
\end{equation*}
and, as $\{u_n\}$ concentrates at most one point, 
\begin{equation}\label{eq:1}
c=\mathcal{J}_\nu(u_n,v_n)+o(1)=\frac{2-s}{2(N-s)} \int_{\R^N} \frac{ u_n^{2_{s}^*}}{|x|^{s}}dx +o(1)\to \frac{2-s}{2(N-s)} \rho_j.
\end{equation}
As the concentration takes place at $0$ or $\infty$, the sequence $\{u_n\}$ is a positive PS sequence for $\mathcal{J}_1(u)$, defined in \eqref{funct:Ji}. By \cite[Theorem 3.1]{LiGuoNiu} we get $\rho_j\ge l [\mathcal{S}(\lambda_1,s)]^{\frac{N-s}{2-s}}$ and, thus, from \eqref{eq:1} we conclude
\begin{equation*}
\begin{split}
c&=\mathcal{J}_\nu(u_n,v_n)+o(1)=\mathcal{J}_1(u_n)+o(1)\to\ell \mathfrak{C}(\lambda_1,s),
\end{split}
\end{equation*}
with $\ell \in \mathbb{N}$ contradicting \eqref{PS1}. As a consequence, $\tilde{v}\gneq 0$ in $\mathbb{R}^N$.
To continue, we prove that $u_n\weakto\tilde{u}$ in $\mathcal{D}^{1,2}(\R^N)$ such that $\tilde{u}\not \equiv 0$. Assuming $\tilde{u}=0$ and arguing as above we find $\tilde{v}=z_{\mu}^{(2)}$ so that, as in \eqref{eq:0}, we get a contradiction with \eqref{PS1}. Thus, $\tilde{u},\tilde{v} \gneq 0$. Next, taking $n\to+\infty$ in the equality
\begin{equation*}
\begin{split}
c=&\ \mathcal{J}_\nu(u_n,v_n)-\frac{1}{2}\left\langle \mathcal{J}_\nu'(u_n,v_n){\big|} (u_n,v_n) \right\rangle + o(1)\\
 =&\  \frac{2-s}{2(N-s)}\left( \int_{\R^N} \frac{ u_n^{2_{s}^*}}{|x|^{s}}dx+\int_{\R^N} \frac{ v_n^{2_{s}^*}}{|x|^{s}}dx\right) +\nu \left( \frac{\alpha+\beta-2}{2} \right)  \int_{\mathbb{R}^N} h(x)\frac{u_n^{\alpha}\, v_n^{\beta}}{|x|^s} dx + o(1),
\end{split}
\end{equation*}
we find, for $j \in \{0, \infty\}$,
\begin{equation}\label{eqlemmaPS10}
\begin{split}
c=&  \  \frac{2-s}{2(N-s)}\left( \int_{\R^N} \frac{ \tilde{u}^{2_{s}^*}}{|x|^{s}}dx+ \rho_j+  \int_{\R^N} \frac{ \tilde{v}^{2_{s}^*}}{|x|^{s}}dx\right)+\nu \left( \frac{\alpha+\beta-2}{2} \right)  \int_{\mathbb{R}^N} h(x)  \frac{\tilde{u}^{\alpha}\, \tilde{v}^{\beta}}{|x|^s} dx.
\end{split}
\end{equation}
On the other hand, as $\left\langle \mathcal{J}_\nu'(u_n,v_n){\big|} (\tilde{u},\tilde{v}) \right\rangle \to 0$ as $n\to\infty$, we also get
\begin{equation*}
\|(\tilde{u},\tilde{v}) \|_{\mathbb{D}}=  \int_{\mathbb{R}^N}\frac{ \tilde{u}^{2_{s}^*}}{|x|^{s}}dx+\int_{\mathbb{R}^N}\frac{ \tilde{v}^{2_{s}^*}}{|x|^{s}}dx  + \nu(\alpha+\beta) \int_{\R^N} h(x) \tilde{u}^\alpha \tilde{v}^\beta dx,
\end{equation*}
which is equivalent to say that $(\tilde{u},\tilde{v}) \in\mathcal{N}_\nu$. Then,
\begin{equation}\label{eqlemmaPS11}
\begin{split}
 \mathcal{J}_\nu(\tilde{u},\tilde{v})=&\ \frac{2-s}{2(N-s)}\left( \int_{\R^N} \frac{ \tilde{u}^{2_{s}^*}}{|x|^{s}}dx+ \int_{\R^N} \frac{ \tilde{v}^{2_{s}^*}}{|x|^{s}}dx\right) + \nu \left( \frac{\alpha+\beta-2}{2} \right)  \int_{\mathbb{R}^N} h\frac{\tilde{u}^{\alpha}\, \tilde{v}^{\beta}}{|x|^s} dx\\
 \le&\ c<\mathfrak{C}(\lambda_1,s)+\mathfrak{C}(\lambda_2,s).
 \end{split}
\end{equation}
Using \eqref{eqlemmaPS10}, \eqref{eqlemmaPS11}, \eqref{Nnueq} \eqref{ineq:con0_a}, \eqref{ineq:con0}, \eqref{ineq:larga} and \eqref{PS1} we get that
\begin{equation*}
\mathcal{J}_\nu(\tilde{u},\tilde{v})= c - \frac{2-s}{(N-s)}\rho_j
<\ \mathfrak{C}(\lambda_1,s)+\mathfrak{C}(\lambda_2,s) - \frac{2-s}{2(N-s)}\left[\mathcal{S}(\lambda_1,s)\right]^{\frac{N-s}{2-s}}=\mathfrak{C}(\lambda_2,s).
\end{equation*}
The above expression implies that
$$
\tilde{c}_\nu= \inf_{(u,v)\in\mathcal{N}_\nu} \mathcal{J}_\nu(u,v) < \mathfrak{C}(\lambda_2,s).
$$
However, for $\nu$ sufficiently small, Theorem~\ref{thm:groundstatesalphabeta} states that $\tilde c_\nu =  \mathfrak{C}(\lambda_2,s)$, which contradicts the former inequality. Thus, we have proved that $u_n\to \tilde{u}$ strongly in $\mathcal{D}^{1,2}(\R^N)$.

\textbf{Case 2}: The sequence $\{u_n\}$ strongly converges to $\tilde{u}$ in $\mathcal{D}^{1,2}(\mathbb{R}^N)$.\hfill\newline
We want to prove that $\{v_n\}$ strongly converges to $\tilde{v}$ in $\mathcal{D}^{1,2}(\mathbb{R}^N)$. By contradiction, suppose that none of its subsequences converge. We start by proving that $\tilde{u}\not \equiv 0$. Assuming that $\tilde{u} \equiv 0$ by contradiction, then $\{v_n\}$ is a PS sequence for the energy functional $\mathcal{J}_2$ defined in \eqref{funct:Ji} at energy level $c$. Since $v_n\weakto\tilde{v}$ in $\mathcal{D}^{1,2}(\R^N)$, then $\tilde{v}$ satisfies the entire problem \eqref{entire}. By \cite[Theorem 3.1]{LiGuoNiu} and \eqref{Jzeta}, one has
\begin{equation*}
\begin{split}
c& = \lim_{n\to+\infty} \mathcal{J}_2 (v_n) =  \mathcal{J}_2 (z_\mu^{(2)})+\ell\mathfrak{C}(\lambda_2,s),
\end{split}
\end{equation*}
for some $\ell \in \mathbb{N}$. If $\tilde{v}\equiv0$ then $c=\ell \mathfrak{C}(\lambda_2,s)$ in contradiction with \eqref{PS2}. If $\tilde{v}\not\equiv0$ then $\tilde{v}=z_\mu^{(2)}$ for some $\mu>0$.  Thus, $c=(\ell+1) \mathfrak{C}(\lambda_2,s)$ in contradiction with \eqref{PS2}. Therefore, we conclude that $\tilde{u} \not \equiv 0$. Conversely, if one assumes that $\tilde v \equiv 0$, then $\tilde{u}$ solves to \eqref{entire}, which implies that $u=z_\mu^{(1)}$ for some $\mu>0$. Therefore, we get
\begin{equation*}
\begin{split}
c&\ge \frac{2-s}{2(N-s)}\left( \int_{\R^N}\frac{ \tilde{u}^{2_{s}^*}}{|x|^{s}} dx +\left[\mathcal{S}(\lambda_2,s)\right]^{\frac{N-s}{2-s}}  \right)= \mathfrak{C}(\lambda_1,s)+\mathfrak{C}(\lambda_2,s),
\end{split}
\end{equation*}
contradicting \eqref{PS1}. Thus, $\tilde{u},\tilde{v}\not \equiv 0$. Using \eqref{eqlemmaPS10} and the assumption that $v_n$ does not strongly converge in $\mathcal{D}^{1,2}(\R^N)$, there exists at least one $k\in \{0, \infty\}$ such that $\overline{\rho}_k>0$, so that
\begin{equation*}
\begin{split}
c=&\ \frac{2-s}{2(N-s)}\left( \int_{\R^N}\frac{ \tilde{u}^{2_{s}^*}}{|x|^{s}} dx +\int_{\mathbb{R}^N}\frac{|\tilde{v}|^{2_{s}^*}}{|x|^{s}}dx +\overline{\rho}_0+\overline{\rho}_{\infty} \right) +\nu \left( \frac{\alpha+\beta-2}{2} \right)  \int_{\mathbb{R}^N} h(x)  \frac{\tilde{u}^{\alpha}\, \tilde{v}^{\beta}}{|x|^s} dx.
\end{split}
\end{equation*}
By using \eqref{eqlemmaPS11}, \eqref{ineq:con0_a}, \eqref{ineq:con0} and \eqref{PS1}, we get
\begin{equation}\label{eqlemmaPS12}
\begin{split}
\mathcal{J}_\nu(\tilde{u},\tilde{v})&= c - \frac{2-s}{2(N-s)}\left( \overline{\rho}_0 + \overline{\rho}_\infty\right) < \mathfrak{C}(\lambda_1,s)+\mathfrak{C}(\lambda_2,s) - \frac{2-s}{2(N-s)}\left[\mathcal{S}(\lambda_2,s)\right]^{\frac{N-s}{2-s}}\\
& =  \mathfrak{C}(\lambda_1,s). 
\end{split}
\end{equation}
By the first equation of \eqref{system:alphabeta} and the definition of the constant $\mathcal{S}(\lambda_1,s)$, it follows that
\begin{equation}\label{eqlemmaPS13}
\begin{split}
\int_{\R^N} \frac{\tilde{u}^{2_{s}^*}}{|x|^{s}} \, dx + \nu \int_{\mathbb{R}^N} h(x)\frac{\tilde{u}^{\alpha}\, \tilde{v}^{\beta}}{|x|^s} dx 
 &\ge \mathcal{S}(\lambda_1,s) \left(\int_{\R^N} \frac{\tilde{u}^{2_{s}^*}}{|x|^{s}} \, dx\right)^{2/2_{s}^*}.
\end{split}
\end{equation}
 Appliying H\"older's inequality as in \eqref{Hold1}, we have
\begin{equation}\label{Holder}
 \int_{\mathbb{R}^N} h(x)\frac{\tilde{u}^{\alpha}\tilde{v}^{\beta}}{|x|^s} dx
  \leq C(h)  \left( \int_{\mathbb{R}^N} \frac{\tilde{u}^{2^*_{s}}}{|x|^{s}} dx\right)^{\frac{\alpha}{2^*_{s}}}  \left( \int_{\mathbb{R}^N} \frac{\tilde{v}^{2^*_{s} }}{|x|^{s}} dx\right)^{\frac{\beta}{2^*_{s}}}. 
\end{equation}
Next, let us take $\displaystyle\sigma_1\vcentcolon= \int_{\R^N} \frac{\tilde{u}^{2_{s}^*}}{|x|^{s}} \, dx$. 
Then, by \eqref{eqlemmaPS11} and \eqref{Holder}, from \eqref{eqlemmaPS13} we get
\begin{equation}\label{eqlemmaPS14}
\sigma_1 + C \nu \sigma_1^{\frac{\alpha}{2_{s}^*}}\ge \mathcal{S}(\lambda_1,s) \sigma_1^\frac{2}{2_{s}^*}.
\end{equation}
On the other hand, since $\tilde{v}\not \equiv 0$, we have, for some $\tilde \varepsilon>0$,
\begin{equation*}
\frac{2-s}{2(N-s)}\int_{\mathbb{R}^N}\frac{ \tilde{v}^{2_{s}^*}}{|x|^{s}} \, dx\ge \tilde \varepsilon.
\end{equation*} 
Taking  $\varepsilon>0$ such that $\tilde \varepsilon\ge \varepsilon \mathfrak{C}(\lambda_1,s)$, by \eqref{eqlemmaPS14} and Lemma~\ref{algelemma}, there exists $\tilde \nu>0$ such that
\begin{equation*}
\sigma_1\ge (1-\varepsilon)[\mathcal{S}(\lambda_1,s)]^{\frac{N-s}{2-s}} \qquad \mbox{ for any } 0<\nu\leq \tilde{\nu}.
\end{equation*}
From previous estimates and \eqref{eqlemmaPS11}, we obtain that
\begin{equation*}
\mathcal{J}_\nu(\tilde{u},\tilde{v}) \ge (1-\varepsilon)\frac{2-s}{2(N-s)}[\mathcal{S}(\lambda_1,s)]^{\frac{N-s}{2-s}} + \tilde{\varepsilon} =\mathfrak{C}(\lambda_1,s),
\end{equation*}
which gives us a contradiction with \eqref{eqlemmaPS12}. Therefore, $v_n\to \tilde{v}$ strongly in $\mathcal{D}^{1,2}(\R^N)$.
\end{proof}
In a similar way we can establish the following.
\begin{lemma}\label{lemmaPS1a}
Assume that $\alpha+\beta<2_{s}^*$, $\beta\ge2$, and $\lambda_1\ge\lambda_2$. Then, there exists $\tilde{\nu}>0$ such that, if $0<\nu\leq\tilde{\nu}$ and $\{(u_n,v_n)\} \subset \mathbb{D}$ is a PS sequence for $\mathcal{J}^+_\nu$ at level $c\in\mathbb{R}$ such that
\begin{equation}\label{PS2a}
\mathfrak{C}(\lambda_2,s)<c<\mathfrak{C}(\lambda_1,s)+\mathfrak{C}(\lambda_2,s),
\end{equation}
and
\begin{equation}\label{PS2aa}
c\neq \ell \mathfrak{C}(\lambda_1,s) \quad \mbox{ for every } \ell \in \mathbb{N}\setminus \{0\},
\end{equation}
then $(u_n,v_n)\to(\tilde{u},\tilde{v}) \in \mathbb{D}$ up to subsequence.
\end{lemma}
\subsection{Critical range $\alpha+\beta=2_s^*$}\hfill\newline
To find minimizing and Mountain--Pass-type solutions in the critical regime we need to extend Lemmas \ref{lemmaPS2} and \ref{lemmaPS1} to the critical regime. This is done in next Lemma \ref{lemcritic}. 
\begin{lemma}\label{lemcritic} 
Assume that  $\alpha+\beta=2_{s}^*$ and hypothesis \eqref{H} holds. Let $\{(u_n,v_n)\} \subset \mathbb{D}$ be a PS sequence for $\mathcal{J}_\nu$ at level $c\in\mathbb{R}$ such that
\begin{itemize}
\item[i)] either $c$ satisfies \eqref{hyplemmaPS2},
\item[ii)] or $c$ satisfies \eqref{PS1} and \eqref{PS2} if $\alpha\ge 2$ and  $\lambda_1\leq\lambda_2$,
\item[iii)] or $c$ satisfies \eqref{PS2a} and \eqref{PS2aa} if $\beta\ge 2$ and  $\lambda_1\ge\lambda_2$.
\end{itemize}
Then, there exists $\tilde{\nu}>0$ such that, for every $0<\nu\leq \tilde{\nu}$, the sequence $(u_n,v_n)\to(\tilde{u},\tilde{v}) \in \mathbb{D}$ up to a subsequence.
\end{lemma}
\begin{proof}
As in the proof of Lemmas \ref{lemmaPS2} and \ref{lemmaPS1}, in order to avoid concentration at the origin, it is enough to prove (see \eqref{eq:concompact1}) that
\begin{equation}\label{radial:at0}
\lim\limits_{\varepsilon\to0}\limsup\limits_{n\to+\infty}\int_{\mathbb{R}^N}h(x)\frac{|u_n|^\alpha |v_n|^{\beta}}{|x|^{s}}\varphi_{0,\varepsilon}(x)dx=0,
\end{equation}
for $\varphi_{0,\varepsilon}$ a smooth cut-off function centered at 0 defined as in \eqref{cutoff}. Analogously, in order to avoid concentration at $\infty$, we have to show that
\begin{equation}\label{radial:atinf}
\lim\limits_{R\to+\infty}\limsup\limits_{n\to+\infty}\int_{|x|>R}h(x)\frac{|u_n|^\alpha |v_n|^{\beta}}{|x|^{s}}\varphi_{\infty,\varepsilon}(x)dx=0,
\end{equation}
where, $\varphi_{\infty,\varepsilon}$ is a cut-off function supported near $\infty$, introduced in \eqref{cutoffinfi}. Let us prove \eqref{radial:at0}. Applying H\"older's inequality as in \eqref{Hold1} and using $\alpha+\beta=2_s^*$, we get
\begin{equation}\label{eq:Hof}
\begin{split}
\int_{\mathbb{R}^N} h\frac{|u_n|^{\alpha}|v_n|^{\beta}}{|x|^s} \varphi_{0,\varepsilon}dx
& \leq\left(\int_{\mathbb{R}^N}h\frac{|u_n|^{2_{s}^*}}{|x|^{s}}\varphi_{0,\varepsilon}dx\right)^{\frac{\alpha}{2_{s}^*}}\left(\int_{\mathbb{R}^N}h\frac{|v_n|^{2_{s}^*}}{|x|^{s}}\varphi_{0,\varepsilon}dx\right)^{\frac{\beta}{2_{s}^*}}.
\end{split}
\end{equation}
 Because of \eqref{con-comp} and \eqref{H} we get
\begin{equation*}
\lim\limits_{n\to+\infty}\int_{\mathbb{R}^N}h(x)\frac{|u_n|^{2_{s}^*}}{|x|^{s}}\varphi_{0,\varepsilon}dx=\int_{\mathbb{R}^N}h(x)\frac{|\tilde{u}|^{2_{s}^*}}{|x|^{s}}\varphi_{0,\varepsilon}\,dx+\rho_0h(0)
\leq\int_{|x|\leq\varepsilon}h(x)\frac{|\tilde{u}|^{2_{s}^*}}{|x|^{s}} \, dx,
\end{equation*}
and 
\begin{equation*}
\lim\limits_{n\to+\infty}\int_{\mathbb{R}^N}h(x)\frac{|v_n|^{2_{s}^*}}{|x|^{s}}\varphi_{0,\varepsilon}dx=\int_{\mathbb{R}^N}h(x)\frac{|\tilde{v}|^{2_{s}^*}}{|x|^{s}}\varphi_{0,\varepsilon}\,dx+\overline{\rho}_0h(0)\leq\int_{|x|\leq\varepsilon}h(x)\frac{|\tilde{v}|^{2_{s}^*}}{|x|^{s}} \, dx,
\end{equation*}
so \eqref{radial:at0} follows. Since $\lim\limits_{|x|\to+\infty}h(x)=0$, the proof of \eqref{radial:atinf} follows similarly.
\end{proof}
\section{Proofs of main Results}\label{section:main}
Once we have ensured the PS condition under a {\it quantization} of the energy levels, we can prove the main results concerning the existence of bound and ground states to \eqref{system:alphabeta}.
\begin{proof}[Proof of Theorem~\ref{thm:nugrande}]
Fixed $(u,v)\in\mathbb{D}\setminus \{(0,0)\}$, we can take $t$ such that $(tu,tv) \in \mathcal{N}_\nu$ with $t$ satisfying \eqref{normH}. Since $\alpha+\beta>2$, then $t=t_\nu \to 0$ as $\nu\to+\infty$. Indeed, by \eqref{normH}, we have
\begin{equation*}
\lim_{\nu \to+\infty} t_\nu^{\alpha+\beta-2} \nu = \dfrac{\|(u,v)\|_\mathbb{D}^2}{ \displaystyle \int_{\mathbb{R}^N} h(x) \dfrac{|u|^{\alpha} |v|^{\beta}}{|x|^s}  \, dx}.
\end{equation*}
This implies that energy of $(t_\nu u ,t_\nu v)$ is 
\begin{equation*}
\mathcal{J}_\nu(t_\nu u ,t_\nu v)=\left(\frac{1}{2}-\frac{1}{\alpha+\beta} +o(1)  \right) t_\nu^2 \|(u,v)\|^2_\mathbb{D}.
\end{equation*}
Next, we can derive that
\begin{equation}\label{minimumlevel}
\tilde{c}_\nu=\inf_{(u,v) \in \mathcal{N}_\nu} \mathcal{J}_\nu (u,v)< \min \{\mathcal{J}_\nu(z_\mu^{(1)},0),\mathcal{J}_\nu(0,z_\mu^{(2)}) \}=\min \{ \mathfrak{C}(\lambda_1,s), \mathfrak{C}(\lambda_2,s)\},
\end{equation}
for some $\nu>\overline{\nu}$ where $\overline{\nu}$ large enough. If $\alpha+\beta < 2^*_{s}$, the existence of $(\tilde{u},\tilde{v}) \in \mathbb{D}$ such that $\mathcal{J}_\nu(\tilde{u},\tilde{v})=\tilde{c}_\nu$ follows by Lemma~\ref{lemmaPS2}. Concerning the positivity of the solution, notice that
\begin{equation*}
\mathcal{J}_\nu(|\tilde{u}|,|\tilde{v}|)= \mathcal{J}_\nu(\tilde{u},\tilde{v}),
\end{equation*}
which allows us to suppose that $\tilde{u}\ge 0$ and $\tilde{v}\ge 0$ in $\mathbb{R}^N$. By using classical regularity arguments, $\tilde{u}$ and $\tilde{v}$ are indeed smooth in $\R^N\setminus\{0\}$. Moreover, $\tilde{u}\not \equiv 0$ and $\tilde{v}\not \equiv 0$. Otherwise, if $\tilde{u}\equiv 0$, then $\tilde{v}\ge 0 $ and $\tilde{v}$ verifies \eqref{entire}, so $\tilde{v}=z_\mu^{(2)}$, which contradicts the energy level assumption \eqref{minimumlevel}. Analogously, we deduce $\tilde{v}\not\equiv 0$. Next, by applying the maximum principle in $\R^N\setminus\{0\}$, one obtains $(\tilde{u},\tilde{v}) \in \mathcal{N}_\nu$ such that $\tilde{u}> 0$ and $\tilde{v}> 0$ in $\mathbb{R}^N\setminus\{0\}$, completing the thesis of this theorem. If $\alpha+\beta=2^*_{s}$, the same conclusion follows by using Lemma~\ref{lemcritic}.
\end{proof}

\begin{proof}[Proof of Theorem~\ref{thm:lambdaground}]
We shall show the existence of a positive ground state by supposing the alternative $i)$. The proof under assumption $ii)$ is analogous. Due to such hypotheses, Proposition~\ref{thmsemitrivialalphabeta} guarantees that  $(z_\mu^{(1)},0)$ is a saddle point of $\mathcal{J}_\nu$ on $\mathcal{N}_\nu$. Moreover,
$$
\tilde{c}_\nu<\mathcal{J}_\nu(z_\mu^{(1)},0)=\min \{\mathfrak{C}(\lambda_1,s),\mathfrak{C}(\lambda_2,s)\},
$$
where $\tilde{c}_\nu $ defined in \eqref{ctilde}. If $\alpha+\beta<2^*_s$, Lemma~\ref{lemmaPS2} ensures the existence  of $(\tilde{u},\tilde{v})\in\mathcal{N}_\nu$ with $\tilde{c}_\nu=\mathcal{J}_\nu(\tilde{u},\tilde{v})$. Reasoning as in the above theorem, we obtain that $(\tilde{u},\tilde{v})$ is a positive ground state of \eqref{system:alphabeta}. If $\alpha+\beta=2^*_s$, the result follows by using  Lemma~\ref{lemcritic}, so there exists $(\tilde{u},\tilde{v})$ of  \eqref{system:alphabeta}. Indeed, they are positive ground states for \eqref{system:alphabeta}.
\end{proof}

\begin{proof}[Proof of Theorem~\ref{thm:groundstatesalphabeta}]
Let us start by proving $i)$. In virtue of Proposition~\ref{thmsemitrivialalphabeta}, $(0,z_\mu^{(2)})$ is a local minimum for $\nu$ small enough. Now, arguing by contradiction, suppose the existence of $\{\nu_n\} \searrow 0$ such that $\tilde{c}_{\nu_n} <  \mathcal{J}_{\nu_n} (0,z_\mu^{(2)})$. Moreover,
\begin{equation}\label{groundstates1}
\tilde{c}_{\nu_n}< \min \{ \mathfrak{C}(\lambda_1,s), \mathfrak{C}(\lambda_2,s) \}= \mathfrak{C}(\lambda_2,s),
\end{equation}
where $\tilde{c}_{\nu}$ given in \eqref{ctilde} with $\nu=\nu_n$. If $\alpha+\beta<2^*_s$, the PS condition holds at level $\tilde{c}_{\nu_n}$ by Lemma~\ref{lemmaPS2}. If $\alpha+\beta=2^*_s$, apply Lemma \ref{lemcritic} for $\nu$ small to infer the same thesis.

Then, there exists $(\tilde{u}_n,\tilde{v}_n) \in \mathbb{D}$ with $\tilde{c}_{\nu_n}=\mathcal{J}_{\nu_n} (\tilde{u}_n,\tilde{v}_n)$. Due to $\mathcal{J}_{\nu_n} (\tilde{u}_n,\tilde{v}_n)=\mathcal{J}_{\nu_n} (|\tilde{u}_n|,|\tilde{v}_n|)$, one can assume that $\tilde{u}_n \ge 0$ and $\tilde{v}_n \ge 0$. Moreover, as we proved in previous results, we conclude that actually $(\tilde{u}_n, \tilde{v}_n)$ is strictly positive in $\mathbb{R}^N\setminus \{0\}$. Now, let us take
\begin{equation*}
\sigma_{1,n}\vcentcolon=\int_{\mathbb{R}^N} \dfrac{\tilde{u}_n^{2^*_{s}}}{|x|^{s}} \, dx \qquad \mbox{ and } \qquad \sigma_{2,n}\vcentcolon=\int_{\mathbb{R}^N} \dfrac{\tilde{v}_n^{2^*_{s}}}{|x|^{s}}\, dx .
\end{equation*}
Note that, by \eqref{Nnueq2}, we have
\begin{equation}\label{groundstates3}
\tilde{c}_{\nu_n} = \mathcal{J}_{\nu_n} (\tilde{u}_n,\tilde{v}_n)= \frac{2-s}{2(N-s)} \left( \sigma_{1,n} + \sigma_{2,n}\right) + \nu_n \left( \frac{\alpha+\beta-2}{2}  \right)\int_{\mathbb{R}^N} h(x)  \,  \dfrac{\tilde{u}_n^{\alpha} \,  \tilde{v}_n^\beta } {|x|^{s}} \, dx  .
\end{equation}
Combining \eqref{groundstates1} and \eqref{groundstates3}, we deduce that
\begin{equation}\label{groundstates4}
\frac{2-s}{2(N-s)}  \left(\sigma_{1,n} + \sigma_{2,n}\right)<\mathfrak{C}(\lambda_2,s)=\frac{2-s}{2(N-s)} \left[\mathcal{S}(\lambda_2,s)\right]^{\frac{N-s}{2-s}}.
\end{equation}
Since $(\tilde{u}_n,\tilde{v}_n)$ solves \eqref{system:alphabeta}, by using the first equation of \eqref{system:alphabeta} together with  \eqref{Slambda}, one gets
\begin{equation}\label{groundstates45}
\mathcal{S}(\lambda_1,s) (\sigma_{1,n})^{\frac{N-2}{N-s}} \leq \sigma_{1,n} +  \nu_n \alpha  \int_{\mathbb{R}^N} h(x)  \dfrac{\tilde{u}_n^{\alpha} \,  \tilde{v}_n^\beta } {|x|^{s}} \, dx  .
\end{equation}
Applying H\"older's inequality as in \eqref{eq:Hof}, we obtain that
\begin{equation*}
\int_{\mathbb{R}^N} h(x)  \,  \dfrac{\tilde{u}_n^{\alpha} \,  \tilde{v}_n^\beta } {|x|^{s}}   \, dx  \leq C(h)  \left( \int_{\mathbb{R}^N} \dfrac{\tilde{u}_n^{2^*_s}}{|x|^{s}}  \, dx   \right)^{\frac{\alpha}{2^*_{s}}}\left( \int_{\mathbb{R}^N} \dfrac{\tilde{v}_n^{2^*_s}}{|x|^{s}}  \, dx    \right)^{\frac{\beta}{2^*_s}}.
\end{equation*}
and, thus, 
\begin{equation*}
\int_{\mathbb{R}^N} h(x)  \,  \dfrac{\tilde{u}_n^{\alpha} \,  \tilde{v}_n^\beta } {|x|^{s}}   \, dx    \leq C_2(h)(\sigma_{1,n})^{\frac{\alpha}{2}\frac{N-2}{N-s}} [\mathcal{S}(\lambda_2,s)]^{\beta \frac{N-2}{2(2-s)}}.
\end{equation*}
We conclude then
\begin{equation*}
\mathcal{S}(\lambda_1,s) (\sigma_{1,n})^{\frac{N-2}{N-s}} < \sigma_{1,n} + C  \nu_n\alpha   \|h\|_{L^{\infty}} (\sigma_{1,n})^{\frac{\alpha}{2}\frac{N-2}{N-s}} [\mathcal{S}(\lambda_2,s)]^{\beta \frac{N-2}{2(2-s)}} .
\end{equation*}
Since $\mathfrak{C}(\lambda_1,s)>\mathfrak{C}(\lambda_2,s)$, there exists $\varepsilon>0$ such that
\begin{equation}\label{groundstates6}
(1-\varepsilon) \mathcal{S}(\lambda_1,s)]^{\frac{N-s}{2-s}} \ge  [\mathcal{S}(\lambda_2,s)]^{\frac{N-s}{2-s}}.
\end{equation}
By applying Lemma~\ref{algelemma} with $\sigma=\sigma_{1,n}$, there exists $\tilde{\nu}=\tilde{\nu}(\varepsilon)>0$ with
\begin{equation*}
\sigma_{1,n}> (1-\varepsilon) [\mathcal{S}(\lambda_1,s)]^{\frac{N-s}{2-s}} \qquad \mbox{ for any } 0<\nu_n<\tilde{\nu}.
\end{equation*}
The above inequality together with \eqref{groundstates6} implies that $\dfrac{2-s}{2(N-s)}\sigma_{1,n}>\mathfrak{C}(\lambda_2,s)$, which clearly contradicts \eqref{groundstates4}. Therefore, for $\nu$ sufficiently small it is satisfied that
\begin{equation}\label{groundstates7}
\tilde{c}_\nu =   \frac{2-s}{2(N-s)} [\mathcal{S}(\lambda_2,s)]^{\frac{N-s}{2-s}}.
\end{equation}
Let $(\tilde{u},\tilde{v})$ be a minimizer of $\mathcal{J}_\nu$. Arguing by contradiction, we can state either $\tilde{u}\equiv 0$ or $\tilde{v}\equiv 0$. Actually, if $v\equiv 0$, then condition \eqref{groundstates7} is violated. So $u\equiv 0$ and $\tilde{v}$ satisfies the equation
\begin{equation*}
-\Delta \tilde{v} - \lambda_2 \frac{\tilde{v}}{|x|^2}=\dfrac{|\tilde{v}|^{2^*_{s}-2}\tilde{v}}{|x|^{s}} \qquad \mbox{ in } \mathbb{R}^N.
\end{equation*}
To finish, we show that $\tilde{v}= \pm z_{\mu}^{(2)}$. Suppose by contradiction that $\tilde{v}$ changes sign so $\tilde{v}^{\pm} \not \equiv 0$ in $\mathbb{R}^N$. Since $(0,\tilde{v}) \in \mathcal{N}_\nu$, then $(0,\tilde{v}^\pm) \in \mathcal{N}_\nu$ and, by \eqref{groundstates3}, we reach a contradiction, namely,
\begin{equation*}
\tilde{c}_{\nu}= \mathcal{J}_\nu (0,\tilde{v}) = \frac{2-s}{2(N-s)} \int_{\mathbb{R}^N} \frac{|\tilde{v}|^{2^*_{s}}}{|x|^{s}}   = \frac{2-s}{2(N-s)} \int_{\mathbb{R}^N} \left( \frac{|\tilde{v}^+|^{2^*_{s}}}{|x|^{s}}  + \frac{|\tilde{v}^-|^{2^*_s}}{|x|^{s}}\right)  > \mathcal{J}_\nu (0,\tilde{v}^+) \ge  \tilde{c}_{\nu}.
\end{equation*}
Then, $(0,\pm z_{\mu}^{(2)})$ is the minimizer of $\mathcal{J}_\nu$ in $\mathcal{N}_\nu$ if $\lambda_1<\lambda_2$. Consequently, under these hypotheses, $(0,z_{\mu}^{(2)})$ is a ground state to \eqref{system:alphabeta}.  We can deduce $ii)$ and $iii)$ analogously.
\end{proof} 

\begin{proof}[Proof of Theorem~\ref{MPgeom}]
Let us prove the thesis assuming condition $i)$, as the proof  follows analogously under hypothesis $ii)$. First, we shall  prove that the energy functional  $\mathcal{J}_\nu^+\Big|_{\mathcal{N}^+_\nu} $ admits a Mountain--Pass geometry. Secondly, we show that the PS condition holds for the Mountain--Pass level. As a consequence, we deduce the existence of $(\tilde{u},\tilde{v})\in\mathbb{D}$ which is a critical point of $\mathcal{J}_\nu^+$ and, therefore, a bound state of \eqref{system:alphabeta}.

\textbf{Step 1:} Let us define the set of paths that connects $(z_{\mu}^{(1)},0)$ to $(0,z_{\mu}^{(2)})$ continuously,
\begin{equation*}
\Psi_\nu = \left\{ \psi(t)=(\psi_1(t),\psi_2(t))\in C^0([0,1],\mathcal{N}^+_\nu): \, \psi(0)=(z_1^{(1)},0) \mbox{ and } \, \psi(1)=(0,z_1^{(2)})\right\},
\end{equation*}
and the Mountain--Pass level
\begin{equation*}
c_{MP} = \inf_{\psi\in\Psi_\nu} \max_{t\in [0,1]} \mathcal{J}^+_{\nu} (\psi(t)).
\end{equation*}
Take $\psi=(\psi_1,\psi_2) \in \Psi_\nu$, then by the identity \eqref{Nnueqp}, we obtain that
\begin{equation}\label{MPgeomp1}
\begin{split}
\|(\psi_1(t),\psi_2(t))\|^2_{\mathbb{D}}=&\ \int_{\mathbb{R}^N} \frac{(\psi_1^+(t))^{2^*_{s}}}{|x|^{s}}dx +\int_{\mathbb{R}^N} \frac{(\psi_2^+(t))^{2^*_{s}}}{|x|^{s}} dx\\
&  + \nu(\alpha+\beta) \int_{\mathbb{R}^N} h(x) \frac{(\psi_1^+(t))^\alpha (\psi_2^+(t))^\beta}{|x|^s} \, dx  ,
\end{split}
\end{equation}
and, using \eqref{Nnueqp1},
\begin{equation}\label{MPgeomp2}
\begin{split}
\mathcal{J}^+_{\nu} (\psi(t)) =&\ \frac{2-s}{2(N-s)}\int_{\mathbb{R}^N} \frac{(\psi_1^+(t))^{2^*_{s}}}{|x|^{s}} +\int_{\mathbb{R}^N} \frac{(\psi_2^+(t))^{2^*_{s}}}{|x|^{s}} dx \\
  &+ \nu\left(\frac{\alpha+\beta-2}{2}\right) \int_{\mathbb{R}^N}h(x)\frac{(\psi_1^+(t))^\alpha (\psi_2^+(t))^\beta}{|x|^s} \, dx.
\end{split}
\end{equation}
Let us take $\sigma(t)=\left(\sigma_1(t),\sigma_2(t)\right)$, with $\displaystyle\sigma_j(t)\vcentcolon=\int_{\mathbb{R}^N} \dfrac{(\psi_j^+(t))^{2^*_{s}}}{|x|^{s}} \, dx$. Then, by \eqref{H-S_lambda} and \eqref{MPgeomp1}, 
\begin{equation}\label{MPgeomp3}
\begin{split}
\mathcal{S}(\lambda_1,s)(\sigma_1(t))^{\frac{N-2}{N-s}}+\mathcal{S}(\lambda_2,s)(\sigma_2(t))^{\frac{N-2}{N-s}}
\leq&\ \|(\psi_1(t)\|^2_{\lambda_1},\psi_2(t))\|^2_{\mathbb{D}}\\
=&\ \sigma_1(t)+\sigma_2(t)\\
&+\nu (\alpha+\beta) \int_{\mathbb{R}^N} h \frac{(\psi_1^+(t))^\alpha (\psi_2^+(t))^\beta}{|x|^s} \, dx.
\end{split}
\end{equation}
Using H\"older's inequality, one can bound the previous integral as
\begin{equation}\label{MPgeomp4}
\int_{\mathbb{R}^N} h \frac{(\psi_1^+(t))^\alpha (\psi_2^+(t))^\beta}{|x|^s} \, dx \leq \nu \|h\|_{L^{\infty}}  (\sigma_1(t))^{\frac{\alpha}{2}\frac{N-2}{N-s}} (\sigma_2(t))^{\frac{\beta}{2}\frac{N-2}{N-s}}. \end{equation}

Note that, from the definition of $\psi$, we have
\begin{equation*}
\sigma(0)=\left(\int_{\mathbb{R}^N} \frac{(z_1^{(1)})^{2^*_{s}}}{|x|^{s}} \, dx,0\right) \quad \mbox{ and } \quad \sigma(1)=\left(0,\int_{\mathbb{R}^N} \frac{(z_1^{(2)})^{2^*_{s}}}{|x|^{s}} \, dx \right).
\end{equation*}
As $\sigma(t)$ is continuous, there exists $\tilde{t}\in(0,1)$ such that $\sigma_1(\tilde{t})=\tilde{\sigma}=\sigma_2(\tilde{t})$. Taking $t=\tilde{t}$ in inequality \eqref{MPgeomp3} and applying \eqref{MPgeomp4}, we have that
\begin{equation*}
\left( \mathcal{S}(\lambda_1,s) +\mathcal{S}(\lambda_2,s)\right) \tilde \sigma^{\frac{2}{2^*_{s}}} \leq 2 \tilde{\sigma} + \nu (\alpha+\beta) \tilde{\sigma}^{\frac{\alpha+\beta}{2^*_{s}}}.
\end{equation*}
Since $\tilde{\sigma}\neq 0$, by Lemma~\ref{algelemma}, for some $\tilde{\nu}>0$ sufficiently small the previous inequality implies 
\begin{equation}\label{MPgeomp5}
\tilde{\sigma}> \left[\frac{ \mathcal{S}(\lambda_1,s)+ \mathcal{S}(\lambda_2,s)}{2}\right]^{\frac{N-s}{2-s}}>  \left[ \mathcal{S}(\lambda_2,s) \right]^{\frac{N-s}{2-s}}  \qquad \mbox{ for every } 0<\nu\le\tilde{\nu},
\end{equation}
where we have used that $\lambda_2>\lambda_1$. As a result, from \eqref{MPgeomp2} and \eqref{MPgeomp5}, we deduce
\begin{equation*}
\max_{t\in[0,1]} \mathcal{J}^+_\nu(\psi(t)) > 2 \frac{2-s}{2(N-s)} \left[ \mathcal{S}(\lambda_2,s) \right]^{\frac{N-s}{2-s}} = 2 \mathfrak{C}(\lambda_2,s)>\mathfrak{C}(\lambda_1,s).
\end{equation*}
Then, $c_{MP}> \mathfrak{C}(\lambda_1,s)=\max\{\mathcal{J}^+_\nu(z_1^{(1)},0),\mathcal{J}^+_\nu(z_1^{(2)},0)\}$. Thus, $\mathcal{J}^+_\nu$ admits a Mountain--Pass structure on $\mathcal{N}_\nu$. 

\textbf{Step 2:} We consider the path $ \psi(t) =( \psi_1(t),   \psi_2(t))=\left((1-t)^{1/2} z_1^{(1)},t^{1/2}z_1^{(2)} \right)$ for $t\in[0,1]$. By the Nehari manifold properties, there exists a positive function $\gamma:[0,1]\mapsto(0,+\infty)$ such that $\gamma \psi \in \mathcal{N}_\nu^+\cap \mathcal{N}_\nu$ for  $t\in[0,1]$. Note that $\gamma(0)=\gamma(1)=1$. As before, we define
\begin{equation*}
\sigma(t)=(\sigma_1(t),\sigma_2(t))=\left(\int_{\mathbb{R}^N}\frac{\left( \gamma \psi_1(t)\right)^{2^*_{s}}}{|x|^{s}} \, dx , \int_{\mathbb{R}^N} \frac{\left( \gamma \psi_2(t)\right)^{2^*_{s}}}{|x|^{s}} \, dx \right).
\end{equation*}
By \eqref{normcrit}, we have that
\begin{equation}\label{Mpgeom7}
\sigma_1(0)=[\mathcal{S}(\lambda_1,s)]^{\frac{N-s}{2-s}}\quad\text{and}\quad \sigma_2(1)=[\mathcal{S}(\lambda_2,s)]^{\frac{N-s}{2-s}}.
\end{equation}
Since $\gamma \psi(t)\in\mathcal{N}^+_\nu\cap \mathcal{N}_\nu$, using \eqref{normH}, we get
\begin{equation*}
\begin{split}
\left\|\left((1-t)^{1/2} z_1^{(1)},t^{1/2}z_1^{(2)} \right)\right\|^2_\mathbb{D} &= (1-t)\sigma_1(0)+t\sigma_2(1) \\
&= \gamma^{2^*_{s}-2}(t)\left((1-t)^{2^*_{s}/2} \sigma_1(0) + t^{2^*_{s}/2}\sigma_2(1)\right)   \\
&\mkern+20mu + \nu (\alpha+\beta)\gamma^{\alpha+\beta-2} (t)(1-t)^{\alpha/2}t^{\beta/2} \int_{\mathbb{R}^N} h(x) \frac{ (z_1^{(1)})^\alpha (z_1^{(2)})^{\beta}}{|x|^s} dx,
\end{split}
\end{equation*}
implying that, for every  $t\in(0,1)$, it holds
\begin{equation}\label{gammabound}
(1-t)\sigma_1(0)+t\sigma_2(1)>\gamma^{2^*_{s}-2}(t)\left((1-t)^{2^*_{s}/2} \sigma_1(0) + t^{2^*_{s}/2}\sigma_2(1)\right).
\end{equation}
As $\gamma \psi \in \mathcal{N}_\nu^+$, we can express the energy level by using \eqref{Nnueq} and bound it by \eqref{gammabound}, so that
\begin{equation}\label{Jgammapsibound}
\begin{split}
\mathcal{J}_\nu^+ (\gamma \psi(t))=&\ \left( \frac{1}{2}-\frac{1}{\alpha+\beta} \right)\|\gamma\psi(t) \|^2_\mathbb{D}\\
&+ \left(\frac{1}{\alpha+\beta}- \frac{1}{2^*_{s}} \right) \gamma^{2^*_{s}}(t)  \int_{\mathbb{R}^N}\frac{(\psi_1(t))^{2^*_{s}}}{|x|^{s}}dx   +\int_{\mathbb{R}^N}\frac{(\psi_2(t))^{2^*_{s}}}{|x|^{s}}dx\vspace{0.3cm} \\
 =&\ \gamma ^2(t)  \left( \frac{1}{2}-\frac{1}{\alpha+\beta} \right)  \left[(1-t) \sigma_1(0) + t \sigma_2(1) \right] \\
 &+ \gamma^{2^*_{s}} \left(\frac{1}{\alpha+\beta}- \frac{1}{2^*_{s}} \right) \left[  (1-t)^{2^*_{s}/2} \sigma_1(0)   +  t^{2^*_{s}/2} \sigma_2(1)\right] \\
<&\ \frac{2-s}{2(N-s)} \gamma^2(t)\left[(1-t)\sigma_1(0) + t  \sigma_2(1) \right].
\end{split}
\end{equation}
From \eqref{gammabound} and \eqref{Jgammapsibound}, we deduce that
\begin{equation*}
g(t)\!\vcentcolon=\!\frac{2-s}{2(N-s)}\!\! \left[ \dfrac{(1-t) \sigma_1(0) + t \sigma_2(1)}{(1\!-\!t)^{2^*_{s}/2} \sigma_1(0)\!+\! t^{2^*_{s}/2} \sigma_2(1)}  \right]^{\frac{2}{2^*_{s}-2}} \left[(1-t)\sigma_1(0) + t  \sigma_2(1) \right]\!\ge\!\max_{t\in[0,1]} \mathcal{J}_\nu^+ (\gamma \psi(t)).
\end{equation*}
Notice that $g$ attains its maximum value at $t=\frac{1}{2}$. Actually, by \eqref{Mpgeom7}, we have
\begin{equation*}
g\left(\dfrac{1}{2}\right)=\frac{2-s}{2(N-s)} \left(\sigma_1(0) +  \sigma_2(1)\right) = \mathfrak{C}(\lambda_1,s)+\mathfrak{C}(\lambda_2,s).
\end{equation*}
Then, using \eqref{Jgammapsibound} and  \eqref{lamdasalphabeta}, we derive that 
$
\mathcal{J}_\nu^+ (\gamma \psi(t))<  \mathfrak{C}(\lambda_1,s)+\mathfrak{C}(\lambda_2,s) < 3 \mathfrak{C}(\lambda_2,s).
$
Consequently,
$\mathfrak{C}(\lambda_2,s)<\mathfrak{C}(\lambda_1,s)<c_{MP} \leq \max_{t\in[0,1]} \mathcal{J}_\nu^+(\gamma \psi(t))<  3 \mathfrak{C}(\lambda_2,s).
$
Then, the  Mountain--Pass level $c_{MP}$ satisfies the assumptions of Lemmas~\ref{lemmaPS1} and \ref{lemcritic}. By the Mountain--Pass Theorem, we can infer the existence of a sequence $\left\{ (u_n,v_n) \right\} \subset \mathcal{N}^+_\nu$ such that 
\begin{equation*}
\mathcal{J}^+(u_n,v_n ) \to c_\nu \quad\text{and}\quad \mathcal{J}^+|_{\mathcal{N}
^+_\nu}(u_n,v_n ) \to 0.
\end{equation*}
If $\alpha+\beta<2^*_s$, by analogous versions of Lemmas~\ref{lemma:PSNehari} and \ref{lemmaPS1} for $\mathcal{J}^+_\nu$, we get $\left\{ (u_n,v_n) \right\} \to ( \tilde{u},\tilde{v} )$. Indeed, $(\tilde{u},\tilde{v} )$ is a critical point of $\mathcal{J}_\nu$ on $\mathcal{N}_\nu$ so it is also a critical point of $\mathcal{J}_\nu$ defined in $\mathbb{D}$. Moreover, $\tilde{u},\tilde{v} \ge 0$ in $\mathbb{R}^N$ and by the maximum principle in $\mathbb{R}^N\setminus \{0\}$ we conclude they are strictly positive.
For assumptions $ii)$, the PS condition follows by Lemma~\ref{lemmaPS1a}. If $\alpha+\beta=2^*_s$, we follow the same approach using  now Lemma~\ref{lemcritic}.
\end{proof}

\begin{center}{\bf Acknowledgements}\end{center} This work has been partially supported by the Madrid Government (Comunidad de Madrid-Spain) under the Multiannual Agreement with UC3M in the line of Excellence of University Professors (EPUC3M23), and in the context of the V PRICIT (Regional Programme of Research and Technological Innovation).\\ R.L-S. is currently supported by the grant Juan de la Cierva Incorporación fellowship (JC2020-046123-I), funded by MCIN/AEI/10.13039/501100011033, and by the European Union Next Generation EU/PRTR. He is also partially supported by Grant PID2021-122122NB-I00 funded by MCIN/AEI/ 10.13039/501100011033 and by “ERDF A way of making Europe”.
\\
 A.O. is partially supported by the Ministry of Economy and Competitiveness of Spain, under research project PID2019-106122GB-I00.
 \end{document}